\documentclass[12pt]{amsart}
\usepackage{fullpage,url,amssymb,enumerate}
\usepackage[all]{xy}

\usepackage[OT2,T1]{fontenc}

\usepackage{color}

\newcommand{\C}{\mathbb{C}}
\newcommand{\F}{\mathbb{F}}
\newcommand{\Fbar}{{\overline{\F}}}
\newcommand{\PP}{\mathbb{P}}
\newcommand{\Q}{\mathbb{Q}}

\newcommand{\Z}{\mathbb{Z}}
\newcommand{\Qbar}{{\overline{\Q}}}
\newcommand{\Ebar}{{\overline{E}}}
\newcommand{\rhobar}{{\overline{\rho}}}
\newcommand{\vv}{v}
\newcommand{\eps}{\varepsilon}
\newcommand{\calD}{\mathcal{D}}
\newcommand{\calE}{\mathcal{E}}
\newcommand{\calK}{\mathcal{K}}
\newcommand{\calO}{\mathcal{O}}
\newcommand{\fp}{\mathfrak{p}}

\DeclareMathOperator{\Aut}{Aut}
\DeclareMathOperator{\can}{can}
\DeclareMathOperator{\Disc}{disc}
\DeclareMathOperator{\Dic}{Dic}
\DeclareMathOperator{\cond}{cond}
\DeclareMathOperator{\Frob}{Frob}
\DeclareMathOperator{\Gal}{Gal}
\DeclareMathOperator{\id}{id}
\DeclareMathOperator{\Ind}{Ind}
\DeclareMathOperator{\Jac}{Jac}
\DeclareMathOperator{\Sel}{Sel}
\DeclareMathOperator{\Spec}{Spec}
\DeclareMathOperator{\GL}{GL}
\DeclareMathOperator{\SL}{SL}

\newcommand{\ab}{{\operatorname{ab}}}
\newcommand{\ns}{{\operatorname{ns}}}
\newcommand{\spl}{{\operatorname{sp}}}
\newcommand{\unr}{{\operatorname{unr}}}
\newcommand{\tors}{{\operatorname{tors}}}

\newcommand{\To}{\longrightarrow}
\newcommand{\pws}[1]{[\![#1]\!]}
\newcommand{\lrs}[1]{(\!(#1)\!)}
\newcommand{\Magma}{{\sf Magma}}

\numberwithin{equation}{section}

\newtheorem{theorem}[equation]{Theorem}
\newtheorem{lemma}[equation]{Lemma}
\newtheorem{corollary}[equation]{Corollary}
\newtheorem{proposition}[equation]{Proposition}

\theoremstyle{definition}

\theoremstyle{remark}
\newtheorem{remark}[equation]{Remark}

\definecolor{darkgreen}{rgb}{0,0.5,0}
\usepackage[
        colorlinks, citecolor=darkgreen,
        backref,
        pdfauthor={Nuno Freitas, Bartosz Naskrecki, Michael Stoll},
        pdftitle={The generalized Fermat equation with exponents 2, 3, n},
]{hyperref}

\usepackage[alphabetic,backrefs,lite]{amsrefs} 

\usepackage{comment}

\begin{document}

\title[The generalized Fermat equation]%
      {The generalized Fermat equation with exponents 2, 3, $n$}

\subjclass[2010]{Primary 11D41; Secondary 11F80, 11G05, 11G07, 11G30, 14G05, 14G25}

\author{Nuno Freitas}
\address{Mathematics Institute,
         University of Warwick,
         Coventry CV4 7AL,
         United Kingdom}
\email{nunobfreitas@gmail.com}

\author{Bartosz Naskr\k{e}cki}
\address{Faculty of Mathematics and Computer Science,
         Adam Mickiewicz University in Pozna\'n,
         Uniwersytetu Pozna\'nskiego 4, 61-614 Pozna\'n, Poland}
\email{nasqret@gmail.com}

\author{Michael Stoll}
\address{Mathematisches Institut,
         Universit\"at Bayreuth,
         95440 Bayreuth, Germany.}
\email{Michael.Stoll@uni-bayreuth.de}
\urladdr{http://www.computeralgebra.uni-bayreuth.de}

\date{June 10, 2019}

\begin{abstract}
  We study the Generalized Fermat Equation $x^2 + y^3 = z^p$, to be solved
  in coprime integers, where $p \ge 7$ is prime. Using modularity and
  level lowering techniques, the problem can be reduced to the determination
  of the sets of rational points satisfying certain $2$-adic and $3$-adic
  conditions on a finite set of twists of the modular curve~$X(p)$.

  We first develop new local criteria to decide if two elliptic curves
  with certain types of potentially good reduction at 2 and 3 can have
  symplectically or anti-symplectically isomorphic $p$-torsion modules.
  Using these criteria we produce the minimal list of twists of~$X(p)$
  that have to be considered, based on local information at $2$ and~$3$; this list
  depends on $p \bmod 24$. Recent
  results on mod~$p$ representations with image in the normalizer of
  a split Cartan subgroup allow us to reduce the list further in some cases.

  Our second main result is the complete solution of the equation when
  $p = 11$, which previously was the smallest unresolved~$p$. One relevant new
  ingredient is the use of the `Selmer group Chabauty' method introduced
  by the third author in recent work, applied in an Elliptic Curve
  Chabauty context, to determine relevant points on~$X_0(11)$ defined
  over certain number fields of degree~$12$. This result is conditional
  on~GRH, which is needed to show correctness of the computation of
  the class groups of five specific number fields of degree~$36$.

  We also give some partial results for the case $p = 13$.
\end{abstract}

\maketitle


\section{Introduction}

This paper considers the Generalized Fermat Equation
\begin{equation} \label{E:GFE1}
  x^2 + y^3 = \pm z^n \,.
\end{equation}
Here $n \ge 2$ is an integer, and we are interested in
\emph{non-trivial primitive integral solutions},
where an \emph{integral solution} is a triple~$(a,b,c) \in \Z^3$
such that $a^2 + b^3 = \pm c^n$; such a solution is \emph{trivial}
if $abc = 0$ and \emph{primitive} if $a$, $b$ and~$c$ are coprime.
If $n$ is odd, the sign can be absorbed into the $n$th power, and there
is only one equation to consider, whereas for even~$n$, the two sign
choices lead to genuinely different equations.

It is known that for $n \le 5$ there are
infinitely many primitive integral solutions, which come in finitely
many families parameterized by binary forms evaluated at pairs of
coprime integers satisfying some congruence conditions, see for
example~\cite{Edwards} for details.
It is also known that for (fixed) $n \ge 6$ there are only finitely
many coprime integral solutions, see~\cite{DarmonGranville} for $n \ge 7$;
the case $n = 6$ reduces to two elliptic curves of rank zero.
Some non-trivial solutions are known for $n \ge 7$, namely (up to sign changes)
\begin{gather*}
  13^2 + 7^3 = 2^9, \quad
  71^2 + (-17)^3 = 2^7, \quad
  21063928^2 + (-76271)^3 = 17^7, \\
  2213459^2 + 1414^3 = 65^7, \quad
  15312283^2 + 9262^3 = 113^7, \\
  30042907^2 + (-96222)^3 = 43^8, \quad
  1549034^2 + (-15613)^3 = -33^8,
\end{gather*}
and for every~$n$, there is the `Catalan solution' $3^2 + (-2)^3 = 1^n$.
It appears likely (and is in fact a special case of the `Generalized Fermat
Conjecture') that these are the only non-trivial primitive integral
solutions for all $n \ge 6$. This has been verified for
$n = 7$~\cite{PSS2007}, $n = 8$~\cites{Bruin1999,Bruin2003},
$n = 9$~\cite{Bruin2005}, $n = 10$~\cites{Brown2012,Siksek2013}
and $n = 15$~\cite{SiksekStoll2014}. Since any integer $n \ge 6$
is divisible by $6$, $8$, $9$, $10$, $15$, $25$ or a prime $p \ge 7$,
it suffices to deal with $n = 25$ and with $n = p \ge 11$ a prime, given
these results.
The case $n = 25$ is considered in ongoing work by the authors of
this paper; the results will be described elsewhere. So we will
from now on assume that $n = p \ge 7$ (or $\ge 11$) is a prime number.

We note that an explicit version of the $abc$~conjecture with a sufficiently
good exponent would give an effective way of obtaining all solutions
to equation~\eqref{E:GFE1}. Namely, suppose that we know $\gamma > 0$
and $\eps < \frac{5}{61}$ such that for all coprime integers $A, B, C$
with $A + B = C$, we have that
\[ \max\{|A|, |B|, |C|\} \le \gamma \Bigl(\prod_{p \mid ABC} p\Bigr)^{1 + \eps} \,, \]
where the product is over the prime divisors of~$ABC$. Assume that $a^2 + b^3 = \pm c^n$
with coprime $a, b, c$ and set $M = \max\{|a|^2, |b|^3, |c|^n\}$.
We then obtain that
\[ M \le \gamma \Bigl(\prod_{p \mid a^2 b^3 c^n} p\Bigr)^{1 + \eps}
     = \gamma \Bigl(\prod_{p \mid a b c} p\Bigr)^{1 + \eps}
     \le \gamma |abc|^{1 + \eps}
     \le \gamma M^{(1/2 + 1/3 + 1/n)(1 + \eps)} \,,
\]
so $M^{1/6 - 1/n - (5/6 + 1/n)\eps} \le \gamma$. Since $\eps < \frac{5}{61}$,
the exponent on the left is positive as soon as $n \ge 11$, and we get an
effective bound on~$M$ in this case. Since equation~\eqref{E:GFE1}
has been solved completely for $n \le 10$, this then would give a
complete solution. Whether this would result in a practical approach
very much depends on the quality of the bound~$\gamma$ in relation to~$\eps$.
(This is well-known to the experts; see for
example~\cite{CohenBook}*{Prop.~14.6.5 and Exercise~2 on p.~493}.)

Our approach follows and refines the arguments of~\cite{PSS2007}
by combining new ideas around the modular method with recent
approaches to the determination of the set of rational points on curves.
We note that the existence of trivial solutions
with $c \neq 0$ and of the Catalan solutions prevents a successful
application of the modular method alone; see the discussion below.
Nevertheless, in the first part of this paper we will apply a refinement
of it to obtain optimal $2$-adic and $3$-adic information,
valid for an arbitrary prime exponent~$p$.
This information is then used as input for global methods in the second part when
tackling concrete exponents. We now give a more detailed description
of these two parts.

\subsection*{The modular method} \strut

The \emph{modular method} for solving Diophantine equations typically
proceeds in the following steps.
\begin{enumerate}[1.]
  \item To a putative solution associate a \emph{Frey elliptic curve}~$E$.
  \item Use modularity and level lowering results to show that (for a suitable
        prime~$p$) the Galois representation on the $p$-torsion~$E[p]$
        is isomorphic to the mod-$\fp$ representation associated to a newform~$f$
        of weight~$2$ and small level~$N$ (for a suitable prime ideal~$\fp$ above~$p$
        of the field of coefficients of~$f$),
        \begin{equation} \label{E:smallLevel}
          \rhobar_{E,p} \simeq \rhobar_{f,\fp} \,.
        \end{equation}
  \item For each of the possible newforms, show that they cannot occur,
        or find all possible associated solutions.
\end{enumerate}

The most challenging step is often the very last, where we want to obtain a
contradiction or list the corresponding solutions. In the proof of Fermat's
Last Theorem (where the modular method was born) we have $N=2$ and there are no
candidate newforms~$f$, giving a simple contradiction. In essentially every other
application of the method there are candidates for~$f$, therefore more work is
needed to complete the argument.
More precisely, we must show, for each newform~$f$ at the concrete small levels,
that $\overline{\rho}_{E,p} \not\simeq \rhobar_{f,\fp}$. This is for example the
obstruction to proving~FLT over~$\Q(\sqrt{5})$; see~\cite{FS3}.
Thus it is crucial to have methods for distinguishing Galois representations.

One such method is known as `the symplectic argument'; it originated in~\cite{HK2002}
and it uses a `symplectic criterion' (see Theorem~\ref{T:KHtrick})
to decide if an isomorphism of the $p$-torsion
of two elliptic curves having a common prime of (potentially) multiplicative reduction
preserves the Weil pairing. In practice, it sometimes succeeds in distinguishing
between the mod~$p$ Galois representations of elliptic curves having at least two
primes of potentially multiplicative reduction.
Extending the symplectic criteria to include elliptic curves with other types of
reduction will clearly allow to attack many more Diophantine equations.
The main challenge in doing this comes from the fact that, in the presence of
potentially good reduction, the inertia action either does not carry enough
information or is hard to describe explicitly.

In the first part of this paper, we prove new symplectic criteria
(Theorems \ref{T:maincrit2} and~\ref{T:maincrit3}) for certain cases of
potentially good reduction at $2$ and~$3$ and apply them to our concrete equation $x^2 + y^3 = z^p$.
Previously the only such criterion available was \cite{HK2002}*{Prop.~A.2}, which
contains a large list of hypotheses making it hard to apply.

\begin{remark}
  While completing this paper, the methods of this first part
  have been generalized in work of the first author~\cites{F33p,FK16};
  furthermore, our symplectic results together with their more general variants have
  already allowed for applications to further Fermat-type equations,
  including the equation $x^3 + y^3 = z^p$; see~\cites{F33p,BBF,FK}.
\end{remark}

Note that equation~\eqref{E:GFE1} admits, for all~$p$, the Catalan solution mentioned
above, but also the \emph{trivial solutions} $(\pm 1,1,0)$, $(\pm 1,0,1)$, $\pm (0,1,1)$.
We remark that the existence of these solutions
is a powerful obstruction to the success of the modular method alone.
Indeed, if by evaluating the Frey curve at these solutions we obtain a non-singular
curve, then we will find the modular form corresponding to it via modularity
among the forms with `small' level. This means that we do obtain a `real'
isomorphism in~\eqref{E:smallLevel}, which for arbitrary~$p$
we cannot discard with current methods (except under highly favorable
conditions). Nevertheless, we will apply our
refinement of the symplectic argument together with a careful
analysis of~\eqref{E:smallLevel} restricted to certain
decomposition groups to obtain finer local information,
valid for an arbitrary exponent~$p$.
This information is then used as input for global methods in
the second part when tackling concrete exponents.

More precisely, our goal is to reduce the study of equation~\eqref{E:GFE1}
to the problem of determining the sets of rational points (satisfying
some congruence conditions at $2$ and~$3$) on a small number of
twists of the modular curve~$X(p)$. For this, we apply our
new symplectic criteria to reduce the list of twists
that have to be considered (in the case of irreducible $p$-torsion on
the Frey elliptic curve, which always holds for $p \neq 7, 13$)
that was obtained in~\cite{PSS2007}.
We also make use of fairly recent results regarding
elliptic curves over~$\Q$ such that the image of the mod~$p$ Galois
representation is contained in the normalizer of a split Cartan
subgroup. Our results here are summarized in Table~\ref{Table-twist},
which says that, depending on the residue class of~$p$ mod~$24$,
there are between four and ten twists that have to be considered.

\subsection*{Rational points on curves} \strut

In the second part of the paper, our main goal is to give a proof of the following.

\begin{theorem}
  Assume the Generalized Riemann Hypothesis.
  Then the only non-trivial primitive integral solutions of the equation
  \begin{equation} \label{E:2311}
    x^2 + y^3 = z^{11}
  \end{equation}
  are the Catalan solutions $(a,b,c) = (\pm 3, -2, 1)$.
\end{theorem}

This appears to be only the second `hyperbolic' instance
(i.e., with $1/p + 1/q + 1/r < 1$)
of the Generalized Fermat Equation with pairwise distinct prime exponents
$p, q, r$ that could be solved completely.
(The first instance was $\{p,q,r\} = \{2,3,7\}$, which was solved
in~\cite{PSS2007}.)

We use several ingredients to obtain this
result. One is the work of Fisher~\cite{Fisher2014}, who obtained
an explicit description of the relevant twists of~$X(11)$
(which we determine in the first part). These curves have genus~$26$
and are therefore not amenable to any direct methods for determining
their rational points. We can (and do) still use Fisher's description
to obtain local information, in particular on the location in~$\Q_2$
of the possible $j$-invariants of the Frey curves. The second
ingredient is the observation that any rational point on one of the relevant twists of~$X(11)$
maps to a point on the elliptic curve~$X_0(11)$ that is defined over
a certain number field~$K$ of degree (at most)~$12$ that only depends
on~$E$ and such that the image of this point under the $j$-map is rational.
This is the setting of `Elliptic Curve Chabauty'~\cite{Bruin2003};
this approach was already taken in an earlier unsuccessful attempt
to solve equation~\eqref{E:2311}
by David Zureick-Brown. To carry this out in the usual way, one needs
to find generators of the group~$X_0(11)(K)$ (or at least of a subgroup
of finite index), which proved to be infeasible in some of the cases.
We get around this problem by invoking the third ingredient, which
is `Selmer Group Chabauty' as described in~\cite{Stoll2017b},
applied in the Elliptic Curve Chabauty setting. We note that we need
the Generalized Riemann Hypothesis (GRH) to ensure the correctness of the
class group computation for the number fields of degree~$36$ arising
by adjoining to~$K$ the $x$-coordinate of a point of order~$2$
on~$X_0(11)$. In principle, the class group can be verified
unconditionally by a finite computation, which, however, would take
too much time with the currently available implementations.
We would like to stress that future improvements of the methods
for computing class groups could result in removing the dependence
of our result from~GRH.

We also give some partial results for $p = 13$, showing that the
Frey curves cannot have reducible $13$-torsion and that the two
CM curves in the list of Lemma~\ref{L:7curves} below can only
give rise to trivial solutions.

We have used the \Magma\ computer algebra system~\cite{Magma} for the
necessary computations.

\subsection*{Acknowledgments} \strut

The work reported on in this paper was supported by the Deutsche
Forschungsgemeinschaft, grant Sto~299/11-1, in the framework of the Priority
Programme SPP~1489. The first author was also partly supported by the grant
{\it Proyecto RSME-FBBVA $2015$ Jos\'e Luis Rubio de Francia}.
We thank David Zureick-Brown for sharing his notes with~us
and Jean-Pierre Serre for pointing us to Ligozat's work
on the splitting of~$J(11)$. We also thank an anonymous referee
for some useful comments.

\subsection*{Notation} \strut

Let $K$ be a field of characteristic zero or a finite field. We write $G_K$ for
its absolute Galois group. If $E/K$ is an elliptic curve, we denote by
$\rhobar_{E,p}$ the Galois representation of~$G_K$ arising from the $p$-torsion
on~$E$. We write $N_E$ for the conductor of $E$ when $K$ is a $p$-adic field or a number field.
For a modular form $f$ and a prime $\fp$ in its field of coefficients we write
$\rhobar_{f,\fp}$ for its associated mod~$\fp$ Galois representation.


\section{Irreducibility and level lowering}
\label{sec:levellowering}

Suppose that $(a,b,c) \in \Z^3$ is a solution to the equation
\begin{equation} \label{E:GFE}
  x^2 + y^3 = z^p, \qquad \text{with $p \geq 7$ prime.}
\end{equation}
Recall that $(a,b,c)$ is \emph{trivial} if $abc=0$ and \emph{non-trivial}
otherwise. An integral solution is \emph{primitive} if $\gcd(a,b,c)=1$ and
\emph{non-primitive} otherwise. Note that equation~\eqref{E:GFE} admits for all~$p$
the trivial primitive solutions $(\pm 1,1,0)$, $(\pm 1,0,1)$, $\pm (0,1,1)$ and the
pair of non-trivial primitive solutions $(\pm 3,2,1)$, which we refer to
as the \textit{Catalan solution(s)}.

As in~\cite{PSS2007}, we can consider a putative solution $(a,b,c)$ of~\eqref{E:GFE}
and the associated Frey elliptic curve
\[ E_{(a,b,c)} \colon y^2 = x^3 + 3bx - 2a \qquad \text{of discriminant $\Delta = -12^3 c^p$.} \]
This curve has invariants
\begin{equation} \label{E:jInvariant}
  c_4 = -12^2 b, \quad c_6 = 12^3 a, \quad  j= \frac{12^3 b^3}{c^p}.
\end{equation}

We begin with a generalization and refinement of Lemma~6.1 in~\cite{PSS2007}.

\begin{lemma} \label{L:7curves}
  Let $p \geq 7$ and let $(a,b,c)$ be coprime integers satisfying $a^2 + b^3 = c^p$ and $c \neq 0$.
  Assume that the Galois representation on $E_{(a,b,c)}[p]$ is irreducible.
  Then there is a quadratic twist $E^{(d)}_{(a,b,c)}$ of~$E_{(a,b,c)}$
  with $d \in \{\pm 1, \pm 2, \pm 3, \pm 6\}$ such that $E^{(d)}_{(a,b,c)}[p]$
  is isomorphic to $E[p]$ as a $G_\Q$-module, where $E$ is one of the following
  seven elliptic curves (specified by their Cremona label):
  \[ 27a1,\; 54a1,\; 96a1,\; 288a1,\; 864a1,\; 864b1,\; 864c1 \,. \]
\end{lemma}

For the convenience of the reader, we give equations of these elliptic curves.
\begin{align*}
   27a1 &\colon y^2 + y = x^3 - 7                & 864a1 &\colon y^2 = x^3 - 3 x + 6 \\
   54a1 &\colon y^2 + x y = x^3 - x^2 + 12 x + 8 & 864b1 &\colon y^2 = x^3 - 24 x + 48 \\
   96a1 &\colon y^2 = x^3 + x^2 - 2 x            & 864c1 &\colon y^2 = x^3 + 24 x - 16 \\
  288a1 &\colon y^2 = x^3 + 3 x \\
\end{align*}

\begin{proof}
  By the proof of~\cite{PSS2007}*{Lemma~4.6}, a twist $E^{(d)}_{(a,b,c)}$
  with $d \in \{\pm 1, \pm 2, \pm 3, \pm 6\}$
  of the Frey curve has conductor dividing~$12^3 N'$, where $N'$ is the
  product of the primes $\ge 5$ dividing~$c$. In fact, carrying out
  Tate's algorithm for $E_{(a,b,c)}$ locally at $2$ and~$3$
  shows that the conductor can be taken to be
  $2^r 3^s N'$ with $r \in \{0,1,5\}$ and $s \in \{1,2,3\}$.
  (This uses the assumption that the solution is primitive.)
  Write~$E$ for this twist~$E^{(d)}_{(a,b,c)}$.

  Using level lowering as in the proof of~\cite{PSS2007}*{Lemma~6.1},
  we find that $\rhobar_{E,p} \simeq \rhobar_{E',p}$
  where $E'$ is an elliptic curve of conductor $27$, $54$, $96$, $288$ or~$864$,
  or else $\rhobar_{E,p} \simeq \rhobar_{f,\fp}$, where $f$ is a newform
  of level~$864$ with field of coefficients $\Q(\sqrt{13})$
  and $\fp \mid p$ in this field.
  Let $f$ be one of these newforms and write
  $\rho = \rho_{f,\fp}|_{D_3}$ for the restriction of the Galois
  representation attached to $f$ to a decomposition group at 3.
  We apply the Loeffler-Weinstein algorithm\footnote{This is implemented
  in \Magma\ via the commands {\tt pi:=LocalComponent(ModularSymbols(f), 3);}
  {\tt WeilRepresentation(pi)}.}
  \cites{LW2012,LW2015} to determine $\rho$ and we obtain $\rho(I_3) \simeq S_3$,
  where $I_3 \subset D_3$ is the inertia group.
  Since $p$ does not divide $6=\#S_3$, we also have that $\rhobar(I_3) \simeq S_3$.
  On the other hand, it is well-known that when $\rhobar_{E,p}(I_3)$ has order~6,
  it must be cyclic (see \cite{Kraus1990}*{page 354}). Thus we cannot have
  $\rhobar_{E,p} \simeq \rhobar_{f,\fp}$ for any of these newforms~$f$.

  We then check that each elliptic curve with conductor
  $27$, $54$, $96$, $288$ or~$864$ is
  isogenous (via an isogeny of degree prime to~$p$)
  to a quadratic twist (with $d$ in the specified set) of one of
  the seven curves mentioned in the statement of the lemma.
\end{proof}

The following proposition shows that the irreducibility assumption in the
previous lemma is automatically satisfied in most cases.

\begin{proposition} \label{P:irred}
  Let $(a,b,c) \in \Z^3$ be a non-trivial primitive solution of \eqref{E:GFE} for
  $p \geq 11$. Write $E = E_{(a,b,c)}$ for the associated Frey curve.
  Then $\rhobar_{E,p}$ is irreducible.
\end{proposition}

\begin{proof}
  First assume $p \neq 13$, so $p = 11$ or $p \ge 17$.
  Then by Mazur's results~\cite{Mazur1978}, there is only
  a finite list of $j$-invariants of elliptic curves over~$\Q$ that have a reducible
  mod~$p$ Galois representation (see also \cite{DahmenPhD}*{Theorem~22}).
  More precisely, either we have that
  \begin{enumerate}[(i)]
    \item $p = 11, 19, 43, 67, 163$ and the corresponding curves
          have integral $j$-invariant, or else
    \item $p = 17$ and the $j$-invariant is $-17^2 \cdot 101^3/2$
          or $-17 \cdot 373^3/2^{17}$.
  \end{enumerate}
  Suppose that $\rhobar_{E,p}$ is reducible, hence the Frey curve $E_{(a,b,c)}$
  corresponds to one of the curves in (i) or (ii). Note that $\gcd(a,b) = 1$.
  Suppose we are in case~(i). Since $p \geq 11$ and the $j$-invariant is integral,
  it follows that $c = \pm 1$, which implies that
  we either have one of the trivial solutions $(\pm 1, 0, 1)$, $\pm(0, 1, 1)$
  or the `Catalan solution' $(\pm 3, -2, 1)$, since the only integral points
  on the elliptic curves $y^2 = x^3 \pm 1$ (which both have finite Mordell-Weil group)
  have $x \in \{0, \pm 1, 2\}$.
  It remains to observe that the Frey curve associated to the Catalan solution
  (which is, up to quadratic twist, $864b1$) is the only curve in its isogeny
  class, so it has irreducible mod~$p$ Galois representations for all~$p$
  (see \cite{LMFDB}*{Elliptic curves over $\Q$}).
  If we are in case~(ii),
  then the $17$-adic valuation of the $j$-invariant contradicts~\eqref{E:jInvariant}.

  For $p = 13$ the claim is shown in Lemma~\ref{L:13nored}.
\end{proof}

We remark that the results of~\cite{PSS2007} show that the statement
of Proposition~\ref{P:irred} is also true for~$p = 7$.

Note that some of the seven curves in Lemma~\ref{L:7curves} are realized by twists of the
Frey curve evaluated at known solutions. Indeed,
\begin{gather*}
  E_{(1,0,1)}^{(6)} = 27a1, \quad E_{(0,1,1)} = 288a1, \quad
  E_{(0,-1,-1)}^{(2)} = 288a2, \quad E_{(3,-2,1)}^{(-2)} = 864b1,
\end{gather*}
and $288a2$ and $288a1$ are related by an isogeny of degree~2.
The solutions $(\pm 1, 1, 0)$ give rise to singular Frey curves.
Note also that $E^{(-d)}_{(-a,b,c)} = E^{(d)}_{(a,b,c)}$, so that $(-1,0,1)$
and $(-3,-2,1)$ do not lead to new curves.


\section{Local conditions and representations of inertia} \label{S:loccon}

Let $\ell$ be a prime.
We write $\Q_\ell^\unr$ for the maximal unramified extension of~$\Q_\ell$
and $I_\ell \subset G_{\Q_\ell}$ for the inertia subgroup.

Let $E$ be an elliptic curve over~$\Q_\ell$ with potentially good reduction.
Let $p \geq 3$, $p \ne \ell$ and $L = \Q_\ell^\unr(E[p])$.
The field~$L$ does not depend on~$p$
and is the smallest extension of~$\Q_\ell^\unr$ over
which $E$ acquires good reduction; see \cite{ST1968}*{\S 2, Corollary 3}.
We call $L$ the \emph{inertial field} of~$E/\Q_\ell$ or of $E$ at $\ell$,
when $E$ is defined over~$\Q$.

We write $L_{2,96}$ and~$L_{2,288}$ for the inertial fields at~$2$
of the elliptic curves with Cremona labels $96a1$ and~$288a1$, respectively,
and we write $L_{3,27}$ and~$L_{3,54}$ for the inertial fields at~$3$ of
the elliptic curves with Cremona labels $27a1$ and~$54a1$. The following
theorem shows that these are all the inertial fields that can arise from
certain types of elliptic curves.

Let $H_8$ denote
the quaternion group and $\Dic_{12} \simeq C_3 \rtimes C_4$ the dicyclic
group of 12~elements. The properties of $H_8$ and of~$\Dic_{12}$ that are
used below can easily be verified using a suitable computer algebra system
like for example~\Magma.

\begin{theorem} \label{T:inertiatypes}
  Let $E/\Q_\ell$ be an elliptic curve with potentially
  good reduction, conductor~$N_E$ and inertial field~$L$.
  Assume further one of the following two sets of hypotheses.
  \begin{enumerate}[\upshape (1)]
    \item \label{inertiatypes-1}
          $\ell = 2$, $\quad \Gal(L/\Q_2^\unr) \simeq H_8 \quad $ and $ \quad \vv_2(N_E) = 5$;
    \item \label{inertiatypes-2}
          $\ell = 3$, $\quad \Gal(L/\Q_3^\unr) \simeq \Dic_{12} \quad $ and $\quad \vv_3(N_E) = 3$.
  \end{enumerate}
  Then $L = L_{2,96}$ or~$L_{2,288}$ in case~\eqref{inertiatypes-1}
  and $L = L_{3,27}$ or~$L_{3,54}$ in case~\eqref{inertiatypes-2}.
\end{theorem}

\begin{proof}
  For a finite extension $K/\Q_\ell$ we denote the Weil subgroup of the absolute
  Galois group~$G_K$ by~$W_K$. We write $W_\ell$ for $W_{\Q_\ell}$.
  We let $r_K \colon K^\times \to W_K^{\ab}$
  denote the reciprocity map from local class field theory.
  It allows us to identify a character~$\chi$ of~$W_K$
  with the character $\chi^A = \chi \circ r_K$ of~$K^\times$.

  There is a representation $\rho_{E} \colon W_\ell \to \GL_2(\C)$
  of conductor~$\ell^{\vv_\ell(N_E)}$ attached to~$E$;
  see \cite{Rohrlich}*{Section~13} and~\cite{DDT}*{Remark~2.14}.
  This representation is either principal series or supercuspidal; see
  \cite{Pacetti}*{Section~2}, noting that the Steinberg representation
  corresponds to infinite inertia.

  Hypotheses \eqref{inertiatypes-1} and~\eqref{inertiatypes-2} both imply
  that $E$ acquires good reduction over a non-abelian extension of~$\Q_\ell^\unr$,
  hence $\rho_E$ is a supercuspidal representation. More precisely,
  $\rho_{E} = \Ind_{W_M}^{W_\ell}\chi$
  where $M/\Q_\ell$ is a quadratic extension, $W_M$
  is the Weil group of~$M$ and $\chi \colon W_M \to \C^\times$
  is a character. If $M/\Q_\ell$ is unramified,
  then $\rho_E|_{I_\ell} \simeq \chi \oplus \chi^s$, where
  $\chi^s(g) := \chi(sgs^{-1})$ and $s \in W_\ell$ lifts the
  non-trivial element of $\Gal(M/\Q_\ell)$. Thus inertia
  has abelian image, a contradiction. Therefore $M/\Q_\ell$ is ramified, and we have
  that $\rho_{E}|_{I_\ell} = \Ind_{I_M}^{I_\ell} (\chi|_{I_M})$,
  where $I_M \subset W_M$ is the inertia subgroup.

  Write $\epsilon_M$ for the quadratic character of~$G_{\Q_\ell}$ fixing~$M$.
  Then $(\chi^A|_{\Q_\ell^\times}) \cdot \epsilon_M^A = \|{\cdot}\|^{-1}$
  as characters of~$\Q_\ell^\times$, where $\|{\cdot}\|$ is the norm character.
  Furthermore, the conductor exponents of $\rho_E$ and~$\chi$ are related by
  $\cond(\rho_E) = \cond(\chi) + \cond(\epsilon_M)$; see \cite{Gerardin}*{2.8}.

  We denote the maximal ideal of~$M$ by~$\fp$.

  Suppose hypothesis~\eqref{inertiatypes-1}. From~\cite{Pacetti}*{Corollary~4.1}
  it follows that $M = \Q_2(d)$, where $d=\sqrt{-1}$ or $d=\sqrt{-5}$,
  hence $\cond(\epsilon_M) = 2$ and $\cond(\chi) = 5 - 2 = 3$.

  Since $\cond(\chi) = 3$, we have that $\chi^A|_{\calO_M^\times}$
  factors via~$(\calO_M/\fp^3)^\times$, which has order~$4$ and is generated by~$2+d$.
  The condition $\chi^A|_{\Z_2^\times} = \epsilon_M^A$ implies $\chi^A(-1)= -1$, thus $\chi^A(2+d)= \pm i$.
  We conclude that there are only two possibilities for~$\chi|_{I_M}$,
  which are related by conjugation, hence giving the same induction~$\rho_{E}|_{I_2}$.
  Thus, for each possible extension~$M$ there is only one field~$L$, so
  there are at most two possible fields~$L$.

  Finally, from \cite{Kraus1990}*{p.~357, Corollary} we see that the curves $96a1$ and~$288a1$
  satisfy hypothesis~\eqref{inertiatypes-1},
  and a direct computation in \Magma\ using the $3$-torsion fields shows that $L_{2,96} \ne L_{2,288}$.
  This proves the theorem in case~\eqref{inertiatypes-1}.

  Now suppose hypothesis~\eqref{inertiatypes-2}. We have $M = \Q_3(d)$, where $d = \sqrt{\pm 3}$,
  both fields satisfying $\vv_3(\Disc(M)) = 1$. Thus $3 = \cond(\chi) + \vv_3(\Disc(M))$
  implies that $\chi$ is of conductor~$\fp^2$.

  For both possible extensions~$M$ the character $\chi^A|_{\calO_M^\times}$
  factors through~$(\calO_M/\fp^2)^\times$,
  which is generated by $-1$ and~$1+d$ of orders $2$ and~$3$, respectively.
  The condition $\chi^A|_{\Z_3^\times} = \epsilon_M^A$ implies that $\chi^A(-1) = -1$
  and the conductor forces $\chi^A(1+d) = \zeta_3^c$ with $c=1$ or~$2$.
  Again, for each possible extension~$M$ there is only one field~$L$, so there are at most two
  possible fields~$L$.

  Finally, from \cite{Kraus1990}*{p.~355, Corollary},
  we see that the curves $27a1$ and~$54a1$ satisfy hypothesis~\eqref{inertiatypes-2},
  and again a direct computation with \Magma\ (but now with $5$-torsion fields as we
  need $\ell \neq p$) shows that $L_{3,27} \neq L_{3,54}$. This concludes the proof.
\end{proof}

In a similar way as in the preceding proof, using \Magma\ to compute with the $3$-torsion
fields over~$\Q_2$, one checks that $96a1$ and~$864c1$
have the same inertial field at~$2$; the same is true for $288a1$, $864a1$ and~$864b1$.
Similarly, working with the $5$-torsion over~$\Q_3$, we also check
that $27a1$, $864b1$ and~$864c1$ have the same inertial field at~$3$;
the same is true for $54a1$ and~$864a1$.

Moreover, we reprove and refine Lemma~\ref{L:7curves} by determining the $2$-adic
and $3$-adic conditions on $a$, $b$ and the twists
$d \in \{\pm 1, \pm 2, \pm 3, \pm 6\}$ such that
the inertial fields at $2$ and~$3$ of $E^{(d)}_{(a,b,c)}$
match those of the seven curves in Lemma~\ref{L:7curves}.
Indeed, the inertial field of $E^{(d)}_{(a,b,c)}$
at~$2$ can be computed from the $3$-torsion field with
only a finite amount of precision
for the Weierstrass model of $E^{(d)}_{(a,b,c)}/\Q_2$.  More precisely, there exists~$k$ such
that if $(x,y,z) \equiv (a,b,c) \bmod 2^k$, then
$E^{(d)}_{(a,b,c)}$ and~$E^{(d)}_{(x,y,z)}$ have the same inertial field at~$2$.
We run over all congruence classes for $x,y,z$ modulo~$2^k$ (with $x,y$ not both even)
and compute the inertial field in each case. We can use a version of Lemma~\ref{L:tech}
to show that $k = 3$ is sufficient.
Analogously we compute the inertial fields at~$3$ using $5$-torsion (here $k = 2$ is sufficient).

The $2$-adic information can be found in Table~\ref{Tab:2adic}.
The last row is interesting: in this case, the twists of the Frey curve that
have good reduction at~$2$ have
trace of Frobenius at~$2$ equal to $\pm 2$, so level lowering can never lead
to a curve of conductor~$27$
(which is the only possible odd conductor dividing~$12^3$), since these curves all have
trace of Frobenius equal to~$0$. The $3$-adic conditions can be found in Table~\ref{Tab:3adic}.
The first column in each table is just a line number; it will be useful as
reference in a later section. The remaining columns contain the indicated data.

\begin{table} \renewcommand{\arraystretch}{1.2}
\[ \begin{array}{|c||c|c||c|c|c|c|} \hline
 i_2  &  a \bmod 4 & b \bmod 8 & d            & \text{curves\large\strut} & v_2(N_E) & j \\
 \hline
 1 &     1         & -1        &  1,-3        & 54a1                & 1 & 2^{6-p} t^{-p} \\
   &    -1         & -1        &  -1,3        & 54a1                & 1 & 2^{6-p} t^{-p} \\\hline
 2 &     0         & 1         & \pm 1, \pm 3 & 288a1, 864a1, 864b1 & 5 & 12^3 - 3 \cdot 2^{10} t^2 \\
   &     0         & -3        & \pm 1, \pm 3 & 288a1, 864a1, 864b1 & 5 & 12^3 + 2^{10} t^2 \\
   &     0         & 3         & \pm 2, \pm 6 & 288a1, 864a1, 864b1 & 5 & 12^3 - 2^{10} t^2 \\
   &     0         & -1        & \pm 2, \pm 6 & 288a1, 864a1, 864b1 & 5 & 12^3 + 3 \cdot 2^{10} t^2 \\\hline
 3 &     2         & 1         & \pm 1, \pm 3 & 96a1, 864c1         & 5 & -2^6 + 2^{11}t \\
   &     2         & -3        & \pm 1, \pm 3 & 96a1, 864c1         & 5 & 15 \cdot 2^6 + 2^{11}t \\
   &     2         & 3         & \pm 2, \pm 6 & 96a1, 864c1         & 5 & 7 \cdot 2^6 + 2^{11} t \\
   &     2         & -1        & \pm 2, \pm 6 & 96a1, 864c1         & 5 & -9 \cdot 2^6 + 2^{11} t \\\hline
 4 &     1         & 0         & -2, 6        & 27a1                & 0 & 2^{15} t^3 \\
   &    -1         & 0         & 2, -6        & 27a1                & 0 & 2^{15} t^3 \\\hline
 5 &    \pm 1      & 2         & \pm 2, \pm 6 & 96a1, 864c1         & 5 & -2^9 + 2^{11} t \\\hline
 6 &    \pm 1      & -2        & \pm 2, \pm 6 & 288a1, 864a1, 864b1 & 5 & 2^9 +2^{11} t \\\hline
 7 &    1          & 4         & -2, 6        & \text{impossible}   & 0 & (2^{12} + 2^{13} t) \\
   &    -1         & 4         & 2, -6        & \text{impossible}   & 0 & (2^{12} + 2^{13} t) \\\hline
 \end{array}
\]
\caption{$2$-adic conditions. Here $E = E^{(d)}_{(a,b,c)}$.\newline
         $j$ gives the possible values of the associated $j$-invariant, with $t \in \Z_2$.}
\label{Tab:2adic}
\end{table}

\begin{table} \renewcommand{\arraystretch}{1.2}
\[ \begin{array}{|c||c|c||c|c|c|c|} \hline
  i_3 & a \bmod 9 & b \bmod 3 & d            & \text{curves\large\strut} & v_3(N_E) & j \\\hline
  1   & 1         & -1        & -3, 6        & 96a1                      & 1 & 3^{3-p} t^{-p}\\
      & -1        & -1        & 3, -6        & 96a1                      & 1 & 3^{3-p} t^{-p} \\\hline
  2   & 0         & 1         & \pm 1, \pm 2, \pm 3, \pm 6 & 288a1       & 2 & 12^3 - 3^7 t^2 \\
      & 0         & -1        & \pm 1, \pm 2, \pm 3, \pm 6 & 288a1       & 2 & 12^3 + 3^7 t^2 \\\hline
  3   & \pm 3     & 1         & \pm 1, \pm 2, \pm 3, \pm 6 & 27a1, 864b1, 864c1 & 3 & 3^3 + 3^6 t \\\hline
  4   & \pm 3     & -1        & \pm 1, \pm 2, \pm 3, \pm 6 & 54a1, 864a1 & 3 & -8 \cdot 3^3 + 3^6 t \\\hline
  5   & \pm 1     & 0         & \pm 1, \pm 2, \pm 3, \pm 6 & 27a1, 864b1, 864c1 & 3 & 3^6 t^3 \\
      & \pm 2     & 0         & \pm 1, \pm 2, \pm 3, \pm 6 & 27a1, 864b1, 864c1 & 3 & 2 \cdot 3^6 t^3 \\\hline
  6   & \pm 4     & 0         & \pm 1, \pm 2, \pm 3, \pm 6 & 288a1       & 2 & 4 \cdot 3^6 t^3 \\\hline
  7   & \pm 1     & 1         & \pm 1, \pm 2, \pm 3, \pm 6 & 54a1, 864a1 & 3 & -4 \cdot 3^3 + 3^5 t \\
      & \pm 4     & 1         & \pm 1, \pm 2, \pm 3, \pm 6 & 54a1, 864a1 & 3 & -3^3 + 3^5 t \\\hline
  8   & \pm 2     & 1         & \pm 1, \pm 2, \pm 3, \pm 6 & 288a1       & 2 & 2 \cdot 3^3 + 3^5 t \\\hline
   \end{array}
\]
\caption{$3$-adic conditions. $j$ is as in Table~\ref{Tab:2adic}, with $t \in \Z_3$ and $E = E^{(d)}_{(a,b,c)}$}
\label{Tab:3adic}
\end{table}

\begin{corollary} \label{C:cond}
  Let $p \ge 11$ be a prime number. Let $(a,b,c) \in \Z^3$ be coprime and satisfy
  $a^2 + b^3 = c^p$. Then $b \not\equiv 4 \bmod 8$, and if $c \neq 0$, then $c$ is not divisible
  by~$6$.
\end{corollary}

\begin{proof}
  Table~\ref{Tab:2adic} shows that $b \equiv 4 \bmod 8$
  is impossible.
  If $c \neq 0$, then we have a twisted Frey curve $E^{(d)}_{(a,b,c)}$, which
  if $6 \mid c$ would have to be $p$-congruent to $54a1$ and to~$96a1$ at the same
  time; this is impossible.
\end{proof}

We observe that the residue classes of $a \bmod 36$ and $b \bmod 24$ determine
the corresponding curve in Lemma~\ref{L:7curves} uniquely, as given in Table~\ref{Tab:curves}.
The line number~$i_2$ of Table~\ref{Tab:2adic} determines the row and the
line number~$i_3$ of Table~\ref{Tab:3adic}, the column.

\begin{table} \renewcommand{\arraystretch}{1.2}
\[ \begin{array}{|r||c|c|c|c|}
     \hline
       i_2 \backslash i_3 & 1 & 2,6,8 &   3,5 & 4,7 \\\hline
               1    &    -  &    -    &    -  &  54a1 \\
               2,6  &    -  & 288a1   & 864b1 & 864a1 \\
               3,5  &  96a1 &    -    & 864c1 &    -  \\
               4    &    -  &    -    &  27a1 &    -  \\\hline
   \end{array}
\]
\caption{Curves $E$ determined by $(a \bmod 36, b \bmod 24)$.}
\label{Tab:curves}
\end{table}

A \Magma\ script that performs the necessary computations for the results
in this section is available as~\texttt{section3.magma} at~\cite{programs}.


\section{Symplectic and anti-symplectic isomorphisms of $p$-torsion}

Let $p$ be a prime.
Let $K$ be a field of characteristic zero or a finite field of characteristic~$\neq p$.
Fix a primitive $p$-th root of unity $\zeta_p \in \bar{K}$.
For $E$ an elliptic curve defined over~$K$ we write $E[p]$ for its $p$-torsion
$G_K$-module, $\rhobar_{E,p} \colon G_K \to \Aut(E[p])$ for the corresponding
Galois representation and $e_{E,p}$ for the Weil pairing on~$E[p]$.
We say that an $\F_p$-basis $(P,Q)$ of $E[p]$ is \emph{symplectic} if
$e_{E,p}(P,Q) = \zeta_p$.

Now let $E /K$ and $E'/K$ be two elliptic curves over some field~$K$ and
let $\phi \colon E[p] \to E'[p]$ be an isomorphism of $G_K$-modules.
Then there is an element $r(\phi) \in \F_p^\times$ such that
\[ e_{E',p}(\phi(P), \phi(Q)) = e_{E,p}(P, Q)^{r(\phi)} \quad \text{for all $P, Q \in E[p]$.} \]
Note that for any $a \in \F_p^\times$ we have $r(a\phi) = a^2 r(\phi)$.
So up to scaling~$\phi$, only the class of~$r(\phi)$ modulo squares matters.
We say that $\phi$ is a \textit{symplectic isomorphism} if
$r(\phi)$ is a square in~$\F_p^\times$, and an \textit{anti-symplectic isomorphism}
if $r(\phi)$ is a non-square.
Fix a non-square~$r_p \in \F_p^\times$.
We say that $\phi$ is \emph{strictly symplectic}, if $r(\phi) = 1$, and
\emph{strictly anti-symplectic}, if $r(\phi) = r_p$.
Finally, we say that $E[p]$ and~$E'[p]$ are
\emph{symplectically} (or \emph{anti-symplectically}) \emph{isomorphic},
if there exists a symplectic (or anti-symplectic) isomorphism of~$G_K$-modules between them.
Note that it is possible that $E[p]$ and~$E'[p]$ are both symplectically
and anti-symplectically isomorphic; this will be the case if and only if
$E[p]$ admits an anti-symplectic automorphism.

Note that an isogeny $\phi \colon E \to E'$ of degree~$n$ not divisible by~$p$ restricts
to an isomorphism $\phi \colon E[p] \to E'[p]$ such that $r(\phi) = n$.
This can be seen from the following computation, where $\hat{\phi}$ is the dual isogeny,
and where we use that fact that $\phi$ and~$\hat{\phi}$ are adjoint with respect
to the Weil pairing.
\[ e_{E',p}(\phi(P), \phi(Q)) = e_{E,p}(P, \hat{\phi}\phi(Q))
                              = e_{E,p}(P, nQ) = e_{E,p}(P,Q)^n .
\]
In particular, $\phi$ induces a symplectic isomorphism on $p$-torsion if $(n/p) = 1$
and an anti-symplectic isomorphism if $(n/p) = -1$.

For an elliptic curve $E/\Q$ there are two modular curves $X_E^+(p) = X_E(p)$ and~$X^-_E(p)$
defined over~$\Q$
that parameterize pairs $(E', \phi)$ consisting of an elliptic curve~$E'$
and a strictly symplectic (respectively, strictly anti-symplectic)
isomorphism $\phi \colon E'[p] \to E[p]$.
These two curves are twists of the standard modular curve~$X(p)$ that classifies
pairs $(E', \phi)$ such that $\phi \colon E'[p] \to M$ is a symplectic isomorphism,
with $M = \mu_p \times \Z/p\Z$ and a certain symplectic pairing on~$M$,
compare~\cite{PSS2007}*{Definition~4.1}.
As explained there, the existence of a non-trivial primitive solution~$(a,b,c)$
of~\eqref{E:GFE} implies that some twisted Frey curve $E_{(a,b,c)}^{(d)}$ gives rise to a
rational point on one of the modular curves $X_E(p)$ or~$X_E^-(p)$, where $E$ is one of
the seven elliptic curves in Lemma~\ref{L:7curves}. Thus the resolution of
equation~\eqref{E:GFE} for any particular $p \ge 11$ is reduced to the determination
of the sets of rational points on 14~modular curves~$X_E(p)$ and~$X^-_E(p)$.

We remark that taking quadratic twists by~$d$ of the pairs $(E', \phi)$
induces canonical isomorphisms $X_{E^{(d)}}(p) \simeq X_E(p)$ and
$X^-_{E^{(d)}}(p) \simeq X^-_E(p)$. Also note that each twist~$X_E(p)$
has a `canonical rational point' representing $(E, \id_{E[p]})$.
On the other hand, it is possible that the twist $X^-_E(p)$ does not
have any rational point. If $E'$ is isogenous to~$E$ by an isogeny~$\phi$
of degree~$n$ prime to~$p$, then $(E', \phi|_{E'[p]})$ gives rise to a rational point
on~$X_E(p)$ when $(n/p) = +1$ and on~$X^-_E(p)$ when $(n/p) = -1$.

In this section we study carefully when isomorphisms of
the torsion modules of elliptic curves preserve the Weil
pairing. This will allow us to discard
some of these 14~modular curves by local considerations.
Of course, from the last paragraph of Section~\ref{sec:levellowering}
it follows that it is impossible to discard
$X_{27a1}(p)$, $X_{288a1}(p)$ or~$X_{864b1}(p)$, since they have rational
points arising from the known solutions. Moreover, if $(2/p) = -1$
we also have a rational point on $ X_{288a1}^{-}(p) \simeq X_{288a2}(p)$.

We now recall a criterion from~\cite{KO} to decide, under certain hypotheses,
whether $E[p]$ and~$E'[p]$ are symplectically isomorphic, which will be useful later.

\begin{theorem}[\cite{KO}*{Proposition~2}] \label{T:KHtrick}
  Let $E$, $E'$ be elliptic curves over $\Q$ with minimal discriminants
  $\Delta$, $\Delta'$. Let $p$ be a prime such that $\rhobar_{E,p} \simeq \rhobar_{E',p}$.
  Suppose that $E$ and~$E'$ have multiplicative reduction at a prime $\ell \ne p$
  and that $p \nmid v_{\ell}(\Delta)$. Then $p \nmid v_{\ell}(\Delta')$,
  and the representations $E[p]$ and $E'[p]$ are symplectically isomorphic
  if and only if $v_\ell(\Delta)/v_\ell(\Delta')$ is a square mod~$p$.
\end{theorem}

The objective of this section is to deduce similar results for
certain types of additive reduction at~$\ell$ (see also \cite{F33p}),
which we will then apply to our Diophantine problem in~Theorem~\ref{T:nominus}.

We will need the following auxiliary result.

\begin{lemma} \label{L:sympcriteria}
  Let $E$ and $E'$ be two elliptic curves defined over a field $K$ and
  with isomorphic $p$-torsion.
  Fix symplectic bases for $E[p]$ and $E'[p]$. Let $\phi \colon E[p] \to E'[p]$ be
  an isomorphism of $G_K$-modules and write $M_\phi$ for the matrix representing~$\phi$ with
  respect to these bases.

  Then $\phi$ is a symplectic isomorphism if and only if $\det(M_\phi)$ is a square mod~$p$;
  otherwise $\phi$ is anti-symplectic.

  Moreover, if $\rhobar_{E,p}(G_K)$ is a non-abelian subgroup of~$\GL_2(\F_p)$,
  then $E[p]$ and~$E'[p]$ cannot be simultaneously symplectically and anti-symplectically isomorphic.
\end{lemma}

\begin{proof}
  This is \cite{F33p}*{Lemma~1}.
\end{proof}


\subsection{A little bit of group theory} \strut

Recall that $H_8$ denotes the quaternion group and $\Dic_{12} \simeq C_3 \rtimes C_4$
is the dicyclic group of 12~elements; these are the two Galois groups occurring
in Theorem~\ref{T:inertiatypes}. We now consider them as subgroups of $\GL_2(\F_p)$.

Write $D_n$ for the dihedral group with $2n$~elements
and $S_n$ and~$A_n$ for the symmetric and alternating groups on $n$~letters.
We write $C(G)$ for the center of a group $G$.
If $H$ is a subgroup of $G$, then we write $N_G(H)$ for its normalizer
and $C_G(H)$ for its centralizer in~$G$.

\begin{lemma} \label{L:normalizer2}
  Let $p \geq 3$ and $G = \GL_2(\F_p)$. Let $H \subset G$ be a subgroup isomorphic to~$H_8$.
  Then the group~$\Aut(H)$ of automorphisms of~$H$ satisfies
  \[ N_G(H)/C(G) \simeq  \Aut(H) \simeq S_4. \]
  Moreover,
  \begin{enumerate}[\upshape(a)]
    \item if $(2/p) = 1$, then all the matrices in~$N_G(H)$ have square determinant;
    \item if $(2/p) = -1$, then the matrices in~$N_G(H)$ with square determinant
          correspond to the subgroup of~$\Aut(H)$ isomorphic to~$A_4$.
  \end{enumerate}
\end{lemma}

\begin{proof}
  There is only one faithful two-dimensional representation of~$H_8$ over~$\F_p$
  ($H_8$ has exactly one irreducible two-dimensional representation and any direct
  sum of one-dimensional representations factors over the maximal abelian quotient),
  so all subgroups~$H$ as in the statement are conjugate. We can therefore assume
  that $H$ is the subgroup generated by
  \[ g_1 = \begin{pmatrix} 0 & -1 \\ 1 & 0 \end{pmatrix} \quad\text{and}\quad
     g_2 = \begin{pmatrix} \alpha & \beta \\ \beta & -\alpha \end{pmatrix},
  \]
  where $\alpha, \beta \in \F_p^\times$ satisfy $\alpha^2 + \beta^2 = -1$.
  It is easy to see that the elements of~$H$ span the $\F_p$-vector space of
  $2 \times 2$ matrices, which implies that $C_G(H) = C(G)$.

  Now the action by conjugation induces a canonical group homomorphism
  $N_G(H) \to \Aut(H)$ with kernel $C_G(H) = C(G)$, leading to an injection
  $N_G(H)/C(G) \to \Aut(H)$. To see that this map is also surjective (and hence
  an isomorphism), note that $N_G(H)$ contains the matrices
  \[ n_1 = \begin{pmatrix} 1 & -1 \\ 1 & 1 \end{pmatrix} \quad\text{and}\quad
     n_2 = \begin{pmatrix} \alpha & \beta-1 \\ \beta+1 & -\alpha \end{pmatrix}
  \]
  and that the subgroup of~$N_G(H)/C(G)$ generated by the images of~$H$ and of
  these matrices has order~$24$. Since it can be easily checked that $\Aut(H_8) \simeq S_4$,
  the first claim follows.

  Note that $A_4$ is the unique subgroup of~$S_4$ of index~2.
  The determinant induces a homomorphism $S_4 \simeq N_G(H)/C(G) \to \F_p^\times/\F_p^{\times 2}$
  whose kernel is either $S_4$ or~$A_4$.
  Since $H \subset \SL_2(\F_p)$ and all matrices in $C(G)$ have square determinant,
  it remains to compute $\det(n_1)$ and~$\det(n_2)$. But $\det(n_1) = 2$ and
  \[ \det(n_2) = -\alpha^2 - (\beta - 1)(\beta + 1) = -\alpha^2 - \beta^2 + 1 = 2 \]
  as well. The result is now clear.
\end{proof}

\begin{lemma} \label{L:normalizer3}
  Let $p \geq 5$ and $G = \GL_2(\F_p)$. Let $H \subset G$ be a subgroup isomorphic to~$\Dic_{12}$.
  Then the group of automorphisms of $H$ satisfies
  \[ N_G(H)/C(G) \simeq \Aut(H) \simeq D_6. \]
  Moreover,
  \begin{enumerate}[\upshape(a)]
    \item if $(3/p) = 1$, then all the matrices in $N_G(H)$ have square determinant;
    \item if $(3/p) = -1$, then the matrices in $N_G(H)$ with square determinant correspond
          to the subgroup of inner automorphisms in~$\Aut(H)$.
  \end{enumerate}
\end{lemma}

\begin{proof}
  The proof is similar to that of Lemma~\ref{L:normalizer2}.
  Again, there is a unique conjugacy class of subgroups isomorphic to~$\Dic_{12}$
  in~$G$, so we can take $H$ to be the subgroup generated by
  \[ g_1 = \begin{pmatrix} \alpha & \beta \\ \beta & 1-\alpha \end{pmatrix} \quad\text{and}\quad
     g_2 = \begin{pmatrix} 0 & -1 \\ 1 & 0 \end{pmatrix},
  \]
  where $\alpha, \beta \in \F_p$ satisfy $\beta^2 = -\alpha^2 + \alpha - 1$ with $\beta \neq 0$.
  As before, one sees that $C_G(H) = C(G)$, so we again have an injective
  group homomorphism $N_G(H)/C(G) \to \Aut(H) \simeq D_6$.

  The normalizer $N_G(H)$ contains the matrix
  \[ M = \begin{pmatrix} 2\alpha-1 & 2\beta \\ 2\beta & 1-2\alpha \end{pmatrix} \]
  and the images of~$H$ and~$M$ generate a subgroup of order~$12$ of~$N_G(H)/C(G)$,
  which shows that the homomorphism is also surjective.

  Since $H \subset \SL_2(\F_p)$, the determinant of any element of~$N_G(H)$
  that induces an inner automorphism of~$H$ is a square. Also, the inner automorphism
  group of~$H$ has order~$6$, so the homomorphism
  $D_6 \simeq N_G(H)/C(G) \to \F_p^\times/\F_p^{\times 2}$ induced by the
  determinant is either trivial or has kernel equal to the group of inner automorphisms.
  This depends on whether the determinant of~$M$,
  \[ \det(M) = -4\alpha^2 + 4\alpha - 1 - 4\beta^2 = 3 , \]
  is a square in~$\F_p$ or not.
\end{proof}


\subsection{The symplectic criteria} \strut
\label{S:theCriteria}

Let $E$, $E'$ be elliptic curves over~$\Q_\ell$ with potentially good reduction
and respective inertial fields $L = \Q_\ell^\unr(E[p])$ and~$L' = \Q_\ell^\unr(E'[p])$.
Suppose that $E[p]$ and~$E'[p]$ are isomorphic as $G_{\Q_\ell}$-modules;
in particular, $L = L'$. Write $I = \Gal(L/\Q_\ell^\unr)$ and recall that $I_\ell$
denotes the inertia subgroup of~$G_{\Q_\ell}$.

If $I$ is not abelian, then
Lemma~\ref{L:sympcriteria} applied with $K=\Q_\ell^\unr$ says that
$E[p]$ and~$E'[p]$ cannot be both symplectically
and anti-symplectically isomorphic $I_{\ell}$-modules.
Since the symplectic type of an isomorphism $\phi \colon E[p] \to E'[p]$
does not depend on whether it is considered as an isomorphism of $G_{\Q_\ell}$-modules
of of $I_\ell$-modules, we can conclude that $E[p]$ and~$E'[p]$ are
symplectically isomorphic as $G_{\Q_\ell}$-modules
if and only if they are symplectically isomorphic as $I_\ell$-modules.
In Theorem~\ref{T:maincrit2} we provide a criterion to decide between the two
possibilities when $\ell = 2$ and $I \simeq H_8$.
In Theorem~\ref{T:maincrit3} we do the same for $\ell = 3$ and $I \simeq \Dic_{12}$.

We now introduce notation and recall facts from \cite{ST1968}*{Section 2}
and~\cite{FK16}.
We note that~\cite{FK16}, which originated as a continuation of the work
done here, contains criteria similar to those stated below,
and also improvements of the second parts of Theorems \ref{T:maincrit2}
and~\ref{T:maincrit3}.

Let $p$ and $\ell$ be primes such that $p \geq 3$ and $\ell \ne p$.
Let $E/\Q_\ell$, $L$ and $I$ be as above.
Write $\Ebar$ for the elliptic curve over $\Fbar_\ell$ obtained by
reduction of a minimal model of $E/L$ and $\varphi \colon E[p] \to \Ebar[p]$ for the
reduction morphism, which preserves the Weil pairing.
Let $\Aut(\Ebar)$ be the automorphism group of~$\Ebar$ over~$\Fbar_\ell$
and write $\psi \colon \Aut(\Ebar) \to \GL(\Ebar[p])$ for the natural
injective morphism. The action of~$I$ on~$L$ induces an injective morphism
$\gamma_E \colon I \to \Aut(\Ebar)$, so that $\Ebar[p]$ is an $I$-module
via $\psi \circ \gamma_E$ in a natural way. Then $\varphi$ is actually
an isomorphism of $I$-modules: for $\sigma \in I$ we have
\begin{equation} \label{eqn:gamma}
 \varphi \circ \rhobar_{E,p}(\sigma) = \psi(\gamma_E(\sigma)) \circ \varphi .
\end{equation}

\begin{theorem} \label{T:maincrit2}
  Let $p \geq 3$ be a prime. Let $E$ and~$E'$ be elliptic curves over~$\Q_2$
  with potentially good reduction. Suppose they have the same inertial field
  and that $I \simeq H_8$.
  Then $E[p]$ and $E'[p]$ are isomorphic as $I_2$-modules. Moreover,
  \begin{enumerate}[\upshape(1)]
    \item if $(2/p) = 1$, then $E[p]$ and $E'[p]$ are symplectically isomorphic $I_2$-modules;
    \item if $(2/p) = -1$, then $E[p]$ and $E'[p]$ are symplectically isomorphic $I_2$-modules
          if and only if $E[3]$ and~$E'[3]$ are symplectically isomorphic $I_2$-modules.
  \end{enumerate}
\end{theorem}

\begin{proof}
  Note that $L = \Q_2^\unr(E[p])$ is the smallest extension of~$\Q_2^\unr$ over which
  $E$ acquires good reduction and that the reduction map~$\varphi$ is an isomorphism
  between the $\F_p$-vector spaces $E[p](L)$ and~$\Ebar[p](\Fbar_2)$.
  By hypothesis $E'$ also has good reduction over~$L$ and $\varphi'$ is an isomorphism.
  Applying equation~\eqref{eqn:gamma} to both $E$ and~$E'$ we see that $E[p]$ and~$E'[p]$
  are isomorphic $I_2$-modules, if we can show that
  $\psi \circ \gamma_E$ and $\psi \circ \gamma_{E'}$ are isomorphic as representations
  into $\GL(\Ebar[p])$ and $\GL(\Ebar'[p])$, respectively.

  Since the map $\gamma_E \colon I \to \Aut(\Ebar)$ is injective, we have
  $H_8 \subset \Aut(\Ebar)$. From \cite{SilvermanI}*{Thm.III.10.1 and Ex.~A.1 on p.~414}
  we see that $\Aut(\Ebar) \simeq \SL_2(\F_3)$ and $j(\Ebar) = 0$. Similarly, we conclude
  that $j(\Ebar') = 0$. Thus, $\Ebar$ and $\Ebar'$ are isomorphic over~$\Fbar_2$.

  So we can fix minimal models of $E/L$ and~$E'/L$
  both reducing to the same $\Ebar$.

  There is only one (hence normal) subgroup~$H$ of~$\SL_2(\F_3)$ isomorphic to~$H_8$.
  Therefore we have that $\psi(\gamma_E(I)) = \psi(\gamma_{E'}(I)) = \psi(H)$ in $\GL(\Ebar[p])$,
  and there must be an automorphism $\alpha \in \Aut(\psi(H))$ such that
  $\psi \circ \gamma_E = \alpha \circ \psi \circ \gamma_{E'}$.
  The first statement of Lemma~\ref{L:normalizer2} shows that there is
  $g \in \GL(\Ebar[p])$ such that $\alpha(x) = gxg^{-1}$ for all $x \in \psi(H)$;
  thus $\psi \circ \gamma_E$ and $\psi \circ \gamma_{E'}$ are isomorphic representations.

  Fix a symplectic basis of $\Ebar[p]$, thus identifying $\GL(\overline{E}[p])$ with $\GL_2(\F_p)$.
  Let $M_g$ denote the matrix representing $g$ and observe that $M_g \in N_{\GL_2(\F_p)}(\psi(H))$.
  Lift the fixed basis to bases of $E[p]$ and $E'[p]$ via the corresponding
  reduction maps $\varphi$ and~$\varphi'$. The lifted bases are symplectic.
  The matrices representing $\varphi$ and~$\varphi'$ with respect to these bases are the
  identity matrix in both cases.
  From \eqref{eqn:gamma} it follows that $\rhobar_{E,p}(\sigma) = M_g \rhobar_{E',p}(\sigma) M_g^{-1}$
  for all $\sigma \in I$. Moreover, $M_g$ represents some $I_2$-module isomorphism
  $\phi \colon E[p] \rightarrow E'[p]$, and from Lemma~\ref{L:sympcriteria}
  we have that $E[p]$ and~$E'[p]$ are symplectically isomorphic if and only if
  $\det(M_g)$ is a square mod~$p$.

  Part~(1) then follows from Lemma~\ref{L:normalizer2}~(a).

  We now prove (2). From Lemma~\ref{L:normalizer2} (b)
  we see that $E[p]$ and~$E'[p]$ are symplectically isomorphic if and only if $\alpha$
  is an automorphism in $A_4 \subset \Aut(\psi(H)) \simeq S_4$.
  Note that these are precisely the inner automorphisms or automorphisms of order~3.
  Note also that all the elements in $S_4$ that are not in~$A_4$ are not inner and have order 2 or~4.
  For each $p$ the map $\alpha_{p} = \psi^{-1}\circ\alpha\circ\psi$ defines an automorphism
  of $\gamma_E(I) = H \subset \Aut(\Ebar)$ satisfying $\alpha_{p} \circ \gamma_{E'} = \gamma_E$.

  We note that the unique automorphism of $\SL_2(\F_3)$ which fixes the
  order~8 subgroup point-wise is the identity. Since $\gamma_E$, $\gamma_{E'}$ are independent of~$p$,
  it follows that $\alpha_{p}$ is the same for all $p$. Since $\alpha$ and $\alpha_p$
  have the same order and are simultaneously inner or not it follows that this property is independent
  of the prime $p$ satisfying $(2/p) = -1$. This shows that $E[p]$ and $E'[p]$
  are symplectically isomorphic $I_2$-modules if and only if $E[\ell]$ and $E'[\ell]$
  are symplectically isomorphic $I_2$-modules for one (hence all) $\ell$ satisfying $(2/\ell) = -1$.
  In particular, we can take $\ell = 3$, and the result follows.
\end{proof}

\begin{theorem} \label{T:maincrit3}
  Let $p \geq 5$ be a prime. Let $E$ and~$E'$ be elliptic curves over~$\Q_3$
  with potentially good reduction.
  Suppose they have the same inertial field and that $I \simeq \Dic_{12}$.
  Then $E[p]$ and~$E'[p]$ are isomorphic as $I_3$-modules. Moreover,
  \begin{enumerate}[\upshape(1)]
    \item if $(3/p) = 1$, then $E[p]$ and $E'[p]$ are symplectically isomorphic $I_3$-modules;
    \item if $(3/p) = -1$, then $E[p]$ and $E'[p]$ are symplectically isomorphic $I_3$-modules
          if and only if $E[5]$ and~$E'[5]$ are symplectically isomorphic $I_3$-modules.
  \end{enumerate}
\end{theorem}

\begin{proof}
  This proof is analogous to the proof of Theorem~\ref{T:maincrit2},
  with 3 and~5 taking over the roles of 2 and~3, respectively.

  In this case $\Aut(\Ebar) \simeq \Dic_{12}$ \cite{SilvermanI}*{Thm.III.10.1},
  so $\psi(\gamma_E(I)) = \psi(\gamma_{E'}(I)) = \psi(\Aut(\Ebar))$.
  We use Lemma~\ref{L:normalizer3} instead of Lemma~\ref{L:normalizer2} to conclude that
  $\alpha$ is given by a matrix~$M_g$. Lemma~\ref{L:normalizer3}~(a) concludes the
  proof of~(1) and Lemma~\ref{L:normalizer3}~(b) the proof of~(2).
\end{proof}


\section{Reducing the number of relevant twists}

Using the results of the previous section we will now show that one can discard some
of the 14 twists of~$X(p)$, depending on the residue class of~$p$ mod~$24$.

\begin{theorem} \label{T:nominus}
  Let $p \ge 11$ be prime and let $(a,b,c) \in \Z^3$ be a non-trivial primitive solution
  of $x^2 + y^3 = z^p$. Then the associated Frey curve~$E^{(d)}_{(a,b,c)}$
  gives rise to a rational point on one of the following twists of~$X(p)$,
  depending on the residue class of~$p$ mod~$24$.
  \[ \renewcommand{\arraystretch}{1.2}
     \begin{array}{|c|l|} \hline
       p \bmod 24 & \text{twists of $X(p)$} \\\hline\hline
        1         & X_{27a1}(p), \; X_{54a1}(p), \; X_{96a1}(p), \; X_{288a1}(p), \;
                    X_{864a1}(p), \; X_{864b1}(p), \; X_{864c1}(p) \\\hline
        5         & X_{27a1}(p), \; X_{54a2}(p), \; X_{96a1}(p), \; X_{288a1}(p), \; X_{288a2}(p), \\
                  & X_{864a1}(p), \; X^-_{864a1}(p), \; X_{864b1}(p), \; X^-_{864b1}(p), \;
                    X_{864c1}(p), \; X^-_{864c1}(p) \\\hline
        7         & X_{27a1}(p), \; X_{54a2}(p), \; X_{96a1}(p), \; X_{288a1}(p), \;
                    X_{864a1}(p), \; X_{864b1}(p), \; X_{864c1}(p) \\\hline
       11         & X_{27a1}(p), \; X_{54a1}(p), \; X_{96a1}(p), \; X_{288a1}(p), \; X_{288a2}(p), \\
                  & X_{864a1}(p), \; X_{864b1}(p), \; X_{864c1}(p) \\\hline
       13         & X_{27a1}(p),                 \; X_{96a2}(p), \; X_{288a1}(p), \; X_{288a2}(p), \;
                    X_{864a1}(p), \; X_{864b1}(p), \; X_{864c1}(p) \\\hline
       17         & X_{27a1}(p), \; X_{54a1}(p),                 \; X_{288a1}(p), \;
                    X_{864a1}(p), \; X_{864b1}(p), \; X_{864c1}(p) \\\hline
       19         & X_{27a1}(p), \; X_{54a1}(p), \; X_{96a2}(p), \; X_{288a1}(p), \; X_{288a2}(p), \\
                  & X_{864a1}(p), \; X^-_{864a1}(p), \; X_{864b1}(p), \; X^-_{864b1}(p), \;
                    X_{864c1}(p), \; X^-_{864c1}(p) \\\hline
       23         & X_{27a1}(p),                                 \; X_{288a1}(p), \;
                    X_{864a1}(p), \; X_{864b1}(p), \; X_{864c1}(p) \\\hline
     \end{array}
  \]
\end{theorem}

\begin{proof}
  Note that among the seven elliptic curves in Lemma~\ref{L:7curves}, $27a1$ and $288a1$ have complex
  multiplication by~$\Z[\omega]$ and~$\Z[i]$, respectively, where $\omega$ is
  a primitive cube root of unity.
  The isogeny classes of the first four curves in the list of
  Lemma~\ref{L:7curves} have the following structure (the edges are labeled by
  the degree of the isogeny):
  \[ \xymatrix@R=5pt{27a2 \ar@{-}[r]^{3} & 27a1 \ar@{-}[r]^{3} & 27a3 \ar@{-}[r]^{3} & 27a4 &
                     54a2 \ar@{-}[r]^{3} & 54a1 \ar@{-}[r]^{3} & 54a3 \\
                     & 96a2 \ar@{-}[dr]^{2} \\
                     & & 96a1 \ar@{-}[r]^{2} & 96a4 & 288a1 \ar@{-}[r]^{2} & 288a2 \\
                     & 96a3 \ar@{-}[ur]_{2}}
  \]
  whereas the isogeny classes of the last three curves are trivial;
  see~\cite{CremonaBook}*{Table~1} or~\cite{LMFDB}.
  Since $27a3$ is the quadratic twist by~$-3$ of~$27a1$, we have that
  $X_{27a1}(p) \simeq X_{27a3}(p)$. If $(3/p) = -1$, then the $3$-isogeny between
  these curves induces an anti-symplectic isomorphism of the mod~$p$ Galois representations,
  and we have that $X_{27a1}(p) \simeq X_{27a3}(p) \simeq X_{27a1}^-(p)$.
  So when $p \equiv 5, 7, 17, 19 \bmod 24$, we only have one twist of~$X(p)$
  coming from~$27a1$. (For the other CM curve~$288a1$, this argument does not
  apply, since it is its own $-1$-twist.)

  For the twists associated to the curves $54a1$ and $96a1$ we apply Theorem~\ref{T:KHtrick}.
  From Table~\ref{Tab:2adic} we see that the Frey curve $E^{(d)}_{(a,b,c)}$
  has multiplicative reduction at $\ell = 2$ if and only if $c$ is even and $d = \pm 1, \pm 3$,
  in which case its minimal discriminant is $\Delta = 2^{-6} 3^3 d^6 c^p$
  (compare the proof of~\cite{PSS2007}*{Lemma~4.6}); in particular, $v_2(\Delta) \equiv -6 \bmod p$.
  Then the Frey curve must be $p$-congruent to $E = 54a1$, which is the only
  curve in our list that has multiplicative reduction at~$2$.
  On the other hand, $\Delta_E = -2^3 3^9$, so that the isomorphism
  between $E^{(d)}_{(a,b,c)}[p]$ and~$E[p]$ is symplectic
  if and only if $(-2/p) = 1$. So for $p \equiv 1, 11, 17, 19 \bmod 24$,
  we get rational points at most on~$X_{54a1}(p)$, whereas for $p \equiv 5, 7, 13, 23 \bmod 24$,
  we get rational points at most on~$X^-_{54a1}(p)$ (which is $X_{54a2}(p)$ when $(3/p) = -1$).
  Similarly, Table~\ref{Tab:3adic} shows that the Frey curve has multiplicative reduction
  at $\ell = 3$ if and only if $c$ is divisible by~$3$. In this case $d = \pm 3, \pm 6$
  and the minimal discriminant is $\Delta = 2^6 3^{-3} c^p$ (see again the proof
  of~\cite{PSS2007}*{Lemma~4.6}), so $v_3(\Delta) \equiv -3 \bmod p$.
  Since $E = 96a1$ is the only curve in our list that has multiplicative reduction
  at~$3$, the Frey curve must be $p$-congruent to it. Since $\Delta_E = 2^6 3^2$,
  we find that the isomorphism between $E^{(d)}_{(a,b,c)}[p]$ and~$E[p]$ is symplectic
  if and only if $(-6/p) = 1$. So for \hbox{$p \equiv 1, 5, 7, 11 \bmod 24$} we get rational
  points at most on~$X_{96a1}(p)$, whereas for $p \equiv 13, 17, 19, 23 \bmod 24$, we get
  rational points at most on~$X^-_{96a1}(p)$ (which is $X_{96a2}$ when $(2/p) = -1$).

  Now we consider the curves~$E$ with conductor at~$2$ equal to $2^5$; these are
  $96a1$, $288a1$, $864a1$, $864b1$ and~$864c1$.
  They all have potentially good reduction at $2$ and $I = \Gal(L/\Q_2^\unr) \simeq H_8$.
  As explained at the beginning of section~\ref{S:theCriteria}, the fact that
  $H_8$ is non-abelian implies that the isomorphism of mod~$p$ Galois representations
  is symplectic if and only if it is symplectic on the level of inertia groups.
  It follows from Theorem~\ref{T:maincrit2}~(1) that in the case that $(2/p) = 1$
  the isomorphism $E^{(d)}_{(a,b,c)}[p] \simeq E[p]$ can only be symplectic.
  So when $p \equiv 1, 7, 17, 23 \bmod 24$, we can exclude the `minus' twists~$X^-_E(p)$
  for $E \in \{96a1, 288a1, 864a1, 864b1, 864c1\}$.

  We can use a similar argument over~$\Q_3$ for the curves~$E$ in our list
  whose conductor at~$3$ is~$3^3$, namely $27a1$, $54a1$, $864a1$, $864b1$
  and~$864c1$. They all have potentially good reduction and $I \simeq \Dic_{12}$.
  By Theorem~\ref{T:maincrit3}~(1) we conclude that the isomorphism
  $E^{(d)}_{(a,b,c)}[p] \simeq E[p]$ must be symplectic when $(3/p) = 1$.
  Thus we can exclude the twists $X^-_E(p)$ for $E$ in the set $\{27a1, 54a1, 864a1, 864b1, 864c1\}$
  when $p \equiv 1, 11, 13, 23 \bmod 24$.

  Finally, from the isogeny diagrams we see that $X_{96a2}(p) \simeq X_{96a1}^{-}(p)$
  and $X_{288a2}(p) \simeq X_{288a1}^{-}(p)$ when $(2/p)=-1$; and also
  $X_{54a2}(p) \simeq X_{54a1}^{-}(p)$ when $(3/p)=-1$.
  This concludes the proof.
\end{proof}

We have already observed that $X_E(p)$ for $E \in \{27a1, 288a1, 288a2, 864b1\}$ always
has a rational point coming from a primitive solution of~\eqref{E:GFE}, so these
twists cannot be excluded. In a similar way, we see that we cannot exclude $X_E(p)$
by local arguments over $\Q_\ell$ with $\ell = 2$ or~$3$, if $E/\Q_\ell$ can be obtained as the Frey curve
evaluated at an $\ell$-adically primitive solution of~\eqref{E:GFE}.
Note that, for $\ell \in \{2,3\}$ and $p \ge 5$,
any $\ell$-adic unit is a $p$-th power in~$\Q_\ell$.
For $\ell = 2$, we have the following triples $(a, b, E)$, where  $a, b \in \Z_2$
are coprime with $a^2 + b^3 \in \Z_2^\times$ and $E/\Q_2$ is the curve obtained as the
associated (local) Frey curve:
\[ (253, -40, 27a2), \quad (10, -7, 96a1), \quad (46, -13, 96a2), \quad (1, 2, 864c1) . \]
For $\ell = 3$, we only obtain $(13, 7, 54a1)$  and $(3, -1, 864a1)$.
The remaining combinations $(E, \ell)$, namely
\begin{gather*}
  (27a2, 3),\; (54a1, 2),\; (54a2, 2),\; (54a2, 3),\; (54a3, 2),\; (54a3, 3),\; (96a1, 3),\; \\
  (96a2, 3),\; (96a3, 2),\; (96a3, 3),\; (96a4, 2),\; (96a4, 3),\; (864a1, 2),\; (864c1, 3),
\end{gather*}
do not arise in this way. This can be verified by checking whether there is $d \in \Q_\ell^\times$
such that $a = c_6(E) d^3$ and $b = -c_4(E) d^2$ are coprime $\ell$-adic integers
such that $a^2 + b^3$ is an $\ell$-adic unit.

In the remainder of this section we will show that there are nevertheless always $2$-adic
and $3$-adic points corresponding to primitive solutions on the twists~$X^{\pm}_E(p)$
listed in Theorem~\ref{T:nominus}. This means that
Theorem~\ref{T:nominus} is the optimal result obtainable from local information
at $2$ and~$3$.

\begin{lemma} \label{L:fine2}
  Let $p \geq 3$ be a prime such that $(2/p) = -1$. Then, up to unramified quadratic twist,
  the $p$-torsion $G_{\Q_2}$-modules of the following curves
  admit exclusively the following isomorphism types:
  \[ 96a1 \stackrel{+}{\simeq} 864c1,\qquad 288a1 \stackrel{-}{\simeq} 864a1,\quad
     288a1 \stackrel{+}{\simeq} 864b1,\quad 864a1 \stackrel{-}{\simeq} 864b1,
  \]
  where $+$ means symplectic and $-$ anti-symplectic.
  Moreover, let $a, b$ be coprime integers satisfying the congruences in line~$i_2$
  of Table~\ref{Tab:2adic} and write $E = E_{(a,b,c)}^{(d)}/\Q_2$, where $d$ is any of
  the possible values in the same line.
  Then, up to unramified quadratic twist, the $p$-torsion $G_{\Q_2}$-modules of the following curves
  admit exclusively the following isomorphism types:
  \[ \begin{array}{l@{\qquad}l@{\quad}l@{\quad}l}
       \text{$i_2 = 2$ with $d = \pm 1, \pm 3$ or $i_2 = 6$} \colon {}
       & E \stackrel{+}{\simeq} 288a1, & E \stackrel{-}{\simeq} 864a1, & E \stackrel{+}{\simeq} 864b1 \\
       \text{$i_2 = 2$ with $d = \pm 2, \pm 6$} \colon {}
       & E \stackrel{-}{\simeq} 288a1, & E \stackrel{+}{\simeq} 864a1, & E \stackrel{-}{\simeq} 864b1 \\
       \text{$i_2 = 3$ with $d = \pm 1, \pm 3$ or $i_2 = 5$} \colon {}
       & E \stackrel{+}{\simeq} 96a1, & E \stackrel{+}{\simeq} 864c1 \\
       \text{$i_2 = 3$ with $d = \pm 2, \pm 6$} \colon {}
       & E \stackrel{-}{\simeq} 96a1, & E \stackrel{-}{\simeq} 864c1
     \end{array}
  \]
  Furthermore, if instead $p \geq 3$ satisfies $(2/p) = 1$,
  then all the previous isomorphisms are symplectic.
\end{lemma}

\begin{proof}
  Let $E$ and $E'$ be any choice of curves that are compared in the statement.
  We have seen in Section~\ref{S:loccon} that $E$ and~$E'$ have the same inertial field at~$2$.
  From Theorem~\ref{T:maincrit2} we know that there is an isomorphism of $I_2$-modules
  $\phi \colon E[p] \to E'[p]$. We can then use~\cite{FK16}*{Theorem~9} to decide
  whether this isomorphism is symplectic or anti-symplectic.
  We will now show that the $I_2$-module isomorphism between $E[p]$ and~$E'[p]$  extends to the whole of~$G_{\Q_2}$ up to unramified quadratic twist.

  Write $L_p = \Q_2(E[p])$ for the $p$-torsion field of $E$ and
  let $U_p$ be the maximal unramified extension of~$\Q_2$ contained in~$L_p$.
  Note that all the curves in the statement acquire good reduction over~$L_3$ and
  have trace of Frobenius $a_{L_3} = -4$ (this can be checked using \Magma).
  Therefore, we know that
  \[ \rhobar_{E,p}|_{I_2} \simeq \rhobar_{E',p}|_{I_2} \quad \text{ and } \quad
     \rhobar_{E,p}|_{G_{L_3}} \simeq \rhobar_{E',p}|_{G_{L_3}}.
  \]

  Note that $G_{U_3} = G_{L_3} \cdot I_2$.
  We now apply \cite{Cent2015}*{Theorem 2} to find that there is a basis in which
  $\rhobar_{E,p}(\Frob_{L_3})$ is the scalar matrix $-2 \cdot \text{Id}_2$.
  Thus the same is true in all bases; therefore $\rhobar_{E,p}(\Frob_{L_3})$
  commutes with all matrices in~$\rhobar_{E,p}(I_2)$.
  Since the same is true for~$E'$, the isomorphism between $\rhobar_{E,p}$ and~$\rhobar_{E',p}$
  on the subgroups $I_2$ and~$G_{L_3}$ extends to~$G_{U_3}$.

  Since all the curves involved have conductor~$2^5$, their discriminants are cubes in $\Q_2$.
  Thus, by~\cite{DDroot2}*{Table~1} we conclude that $\Gal(L_3/\Q_2)$ is isomorphic
  to the semi-dihedral group with 16 elements, hence $H_8 \simeq I \subset G_3$ with index 2,
  thus $[U_3 : \Q_2] = 2$. Therefore, because the representations $\rhobar_{E,p}$
  and~$\rhobar_{E',p}$ are irreducible, they differ at most by the
  quadratic character $\chi$ fixing~$U_3$, that is, we have
  \[ \rhobar_{E,p} \simeq \rhobar_{E',p} \quad \text{ or } \quad
     \rhobar_{E,p} \simeq \rhobar_{E',p} \otimes \chi.
  \]

  The last statement follows from Theorem~\ref{T:maincrit2}~(1).
\end{proof}

When $(E,E')$ is any of the pairs of curves in the first part of the statement
of Lemma~\ref{L:fine2}, the unramified quadratic twist is actually never necessary,
whereas it is necessary for some of the Frey curves in the second part.
This can be checked on the $3$-torsion by an explicit computation;
the result for arbitrary~$p$ follows from this.

Since the isomorphism class of~$X_E(p)$ (or~$X^-_E(p)$) depends only on the
symplectic Galois module~$E[p]$ up to quadratic twist,
Lemma~\ref{L:fine2} implies that over~$\Q_2$,
$X^{\pm}_{96a1}(p) \simeq X^{\pm}_{864c1}(p)$
and $X^{\pm}_{288a1}(p) \simeq X^{\pm}_{864b1}(p)$ (writing $X^+_E(p) = X_E(p)$),
and also that $X^\pm_{288a1}(p) \simeq X^{\pm}_{864a1}(p)$ when $(2/p) = 1$,
whereas $X^\pm_{288a1}(p) \simeq X^{\mp}_{864a1}(p)$ when $(2/p) = -1$.
In this latter case, the Frey curve gives rise to `primitive' $2$-adic points
on $X^-_E(p)$ for $E \in \{96a1, 288a1, 864a1, 864b1, 964c1\}$.

\begin{lemma} \label{L:fine3}
  Let $p \geq 5$ be a prime such that $(3/p) = -1$.
  Then, up to unramified quadratic twist, the $p$-torsion $G_{\Q_3}$-modules of the following curves
  admit exclusively the following isomorphism types:
  \[ 27a1 \stackrel{+}{\simeq} 864c1,\quad 27a1 \stackrel{-}{\simeq} 864b1,\quad
     864b1 \stackrel{-}{\simeq} 864c1,\qquad 54a1 \stackrel{-}{\simeq} 864a1,
  \]
  where $+$ means symplectic and $-$ anti-symplectic.
  Moreover, let $a,b$ be coprime integers satisfying the congruences in line~$i_3$ of Table~\ref{Tab:3adic}
  and write $E = E_{(a,b,c)}^{(d)}/\Q_3$, where $d$ is any of
  the possible values in the same line.
  Then, up to unramified quadratic twist, the $p$-torsion $G_{\Q_3}$-modules of the following curves
  admit exclusively the following isomorphism types:
  \[ \begin{array}{l@{\qquad}l@{\quad}l@{\quad}l}
       \text{$i_3 = 3$ or $5$ with $d = \pm 1, \pm 2$} \colon {}
       & E \stackrel{-}{\simeq} 27a1, & E \stackrel{+}{\simeq} 864b1, & E \stackrel{-}{\simeq} 864c1 \\
       \text{$i_3 = 3$ or $5$ with $d = \pm 3, \pm 6$} \colon {}
       & E \stackrel{+}{\simeq} 27a1, & E \stackrel{-}{\simeq} 864b1, & E \stackrel{+}{\simeq} 864c1 \\
       \text{$i_3 = 4$ or $7$ with $d = \pm 1, \pm 2$} \colon {}
       & E \stackrel{-}{\simeq} 54a1, & E \stackrel{+}{\simeq} 864a1 \\
       \text{$i_3 = 4$ or $7$ with $d = \pm 3, \pm 6$} \colon {}
       & E \stackrel{+}{\simeq} 54a1, & E \stackrel{-}{\simeq} 864a1
     \end{array}
  \]
  Furthermore, if instead $p \geq 5$ satisfies $(3/p) = 1$,
  then all the previous isomorphisms are symplectic.
\end{lemma}

\begin{proof}
  The proof proceeds in the same way as for the previous lemma, replacing $2$ and~$3$ by $3$ and~$5$,
  respectively. We now use Theorem~11 of~\cite{FK16} to obtain the result on the level of inertia.
  We then see that $G_{U_3} = G_{L_5} \cdot I_3$
  and that all the curves in the statement acquire good reduction over~$L_5$ and have
  trace of Frobenius $a_{L_5} = -18$. We conclude as before that $\rhobar_{E,p}$
  and~$\rhobar_{E',p}$ are isomorphic when restricted to~$G_{U_5}$.
  In the present case we have $[U_5 : \Q_3] = 4$, so it a priori conceivable
  that the representations differ by an unramified quartic twist. However, this is not possible,
  because both representations have the same determinant. We conclude that
  they differ at most by an unramified quadratic twist.
\end{proof}

The statements of this lemma can be translated in terms of isomorphisms
over~$\Q_3$ and `primitive' $\Q_3$-points in the same way as for the previous lemma.

These results already show that all the curves $X^{\pm}_E(p)$ listed
in Theorem~\ref{T:nominus} have `primitive' $2$-adic
and $3$-adic points (and therefore cannot be ruled out by local considerations
at $2$ and~$3$), with the possible exception of $2$-adic points on~$X^{\pm}_{54a1}(p)$
and $3$-adic points on~$X^{\pm}_{96a1}(p)$. The next proposition and corollary
show that these curves also have these local `primitive' points. This then implies
that the information in Theorem~\ref{T:nominus} is optimal in the sense that we
cannot exclude more of the twists using purely $2$-adic and $3$-adic arguments.
Note that in some cases it is possible to use local arguments at other primes
to rule out some further twists. For example, in~\cite{FK16}*{Section~32} it is
shown that the twist $X^-_{864a1}(p)$ can be excluded for $p = 19$ and~$43$
and that $X^-_{864b1}(p)$ can be excluded for $p = 19$, $43$ and~$67$,
working at a suitable prime of good reduction.

\begin{proposition} \label{P:Tate}
  Let $\ell \ne p$ be primes with $p \ge 3$ and $\ell \not\equiv 1 \bmod p$.
  Let $E_1$ and $E_2$ be Tate curves over~$\Q_\ell$ with Tate parameters $q_1$ and~$q_2$.
  Write $e_1 = v_\ell(q_1)$ and $e_2 = v_\ell(q_2)$ and suppose that $p \nmid e_1 e_2$.

  Then $E_1[p]$ and $E_2[p]$  are isomorphic $G_{\Q_\ell}$-modules.
\end{proposition}

\begin{proof}
  Fix a primitive $p$th root of unity~$\zeta \in \Qbar_\ell$.
  Since $p \nmid e_1 e_2$, we can find integers $n$ and~$m$ (with $p \nmid n$) satisfying
  $e_2 = n e_1 + p m$. Write $a = q_2/q_1^n \ell^{mp}$; then $a$ is a unit in~$\Q_\ell$.
  Since by assumption $p \neq \ell$ and $\ell \not\equiv 1 \bmod p$, every unit is a $p$th power,
  hence there is $\alpha \in \Q_\ell$ satisfying $\alpha^p = a$.
  Thus $q_2 = q_1^n (\ell^{m} \alpha)^p$ with $p \nmid n$.
  Fix $\gamma_1 \in \Qbar_\ell$ with $\gamma_1^p = q_1$. Setting
  $\gamma_2 = \gamma_1^n \ell^m \alpha$, we have $\gamma_2^p = q_2$.
  By the theory of the Tate curve, we can use $(\zeta q_i^\Z, \gamma_i q_i^\Z)$
  as an $\F_p$-basis for the $p$-torsion of~$E_i$. We claim that the isomorphism
  $\phi \colon E_1[p] \to E_2[p]$ of $\F_p$-vector spaces given by the matrix
  \[ \begin{pmatrix} n & 0 \\ 0 & 1 \end{pmatrix} \in \GL_2(\F_p) \]
  with respect to these bases is actually an isomorphism of $G_{\Q_\ell}$-modules.
  To see this, consider $\sigma \in G_{\Q_\ell}$. Then $\sigma(\zeta) = \zeta^r$
  for some $r \in \F_p^\times$ and $\sigma(\gamma_1) = \zeta^s \gamma_1$ for
  some $s \in \F_p$, which implies that $\sigma(\gamma_2) = \zeta^{ns} \gamma_2$. We then have
  \[ \phi(\sigma(\zeta q_1^{\Z})) = \phi(\zeta^r q_1^{\Z}) = \zeta^{nr} q_2^{\Z}
      = \sigma(\zeta^n q_2^{\Z}) = \sigma(\phi(\zeta q_1^{\Z}))
  \]
  and
  \[ \phi(\sigma(\gamma_1 q_1^{\Z})) = \phi(\zeta^s \gamma_1 q_1^{\Z}))
      = \zeta^{ns} \gamma_2 q_2^{\Z} = \sigma(\gamma_2 q_2^{\Z}) = \sigma(\phi(\gamma_1 q_1^{\Z}))
  \]
  as desired.
\end{proof}

\begin{corollary}
  For $p \geq 5$, there are primitive $2$-adic points on $X^\pm_{54a1}(p)$
  and primitive $3$-adic points on~$X^{\pm}_{96a1}(p)$.
  The signs~$\pm$ here are as given by the entries in Table~\ref{Table-twist}
  (which for the curves considered here summarizes Theorem~\ref{T:nominus}).
\end{corollary}

\begin{proof}
  Note that $2 \not\equiv 1 \bmod p$ and $3 \not\equiv 1 \bmod p$, so
  Proposition~\ref{P:Tate} applies for $\ell = 2$ and~$\ell = 3$.

  Let $W$ denote the curve $54a1$.
  From Table~\ref{Tab:2adic} we see that for the Frey curve $E = E_{a,b,c}$ to be
  $p$-congruent to $W$ we must have $v_2(c) > 0$ and $v_2(a) = \upsilon_2(b)= 0$.
  Note that we can always find $a, b, c \in \Q_2$ satisfying the previous conditions and
  $a^2 + b^3 = c^p$.

  Up to unramified quadratic twist the curves $W/\Q_2$ and $E/\Q_2$ are Tate curves
  with parameters $q_W$ and $q_E$ respectively. We have
  $v_2(q_W) = v_2(\Delta_W) = 3$ and $v_2(q_E) = -v_2(j_E)= -6 + p v_2(c)$.

  Since $p \ne 3$, from the previous proposition we conclude
  that (up to quadratic twist) $E[p]$ and~$W[p]$ are isomorphic $G_{\Q_2}$-modules.
  Therefore we get $2$-adic points on~$X_{54a1}^+(p)$ or~$X_{54a1}^-(p)$ according
  to the signs in Table~\ref{Table-twist}.

  For the curve $96a1$ we argue in the same way, but over $\Q_3$ instead of~$\Q_2$.
\end{proof}

A \Magma\ script that performs the necessary computations for the results
in this section is available as~\texttt{section5.magma} at~\cite{programs}.


\section{Ruling out twists coming from CM curves}

In~\cite{BPR2013}*{Corollary~1.2} it is shown that for $p \ge 11$, $p \neq 13$,
the image of the mod~$p$ Galois representation of any elliptic curve~$E$ over~$\Q$
is never contained in the normalizer of a split Cartan subgroup unless $E$ has
complex multiplication. This allows us to deduce the following.

\begin{lemma} \label{L:CM}
  Let $p \ge 17$ be a prime number.
  \begin{enumerate}[\upshape(1)]
    \item If $p \equiv 1 \bmod 3$, then the only primitive solutions of~\eqref{E:GFE}
          coming from rational points on~$X^{\pm}_{27a1}(p)$ are the trivial
          solutions $(\pm 1)^2 + 0^3 = 1^p$.
    \item If $p \equiv 1 \bmod 4$, then the only primitive solutions of~\eqref{E:GFE}
          coming from rational points on~$X^{\pm}_{288a1}(p)$ are the trivial
          solutions $0^2 + (\pm 1)^3 = (\pm 1)^p$ (with the same sign on both sides).
  \end{enumerate}
\end{lemma}

\begin{proof}
  If a primitive solution $(a,b,c)$ gives rise to a Frey curve~$E'$ such that
  $E'[p] \cong E[p]$ for $E = 27a1$, then the image of Galois in $\GL(E'[p]) \cong \GL(E[p])$
  is contained in the normalizer of a split Cartan subgroup, since $E$ has complex
  multiplication by~$\Z[\omega]$ and $p$ splits in this ring when $p \equiv 1 \bmod 3$.
  It follows that $E'$ also has complex multiplication, which implies that $c = \pm 1$.
  Since the Frey curve of the Catalan solution does not have CM, the solution must
  be trivial, and then only the given solution corresponds to the right curve~$E$.
  The other case is similar, using the fact that $288a1$ has CM by~$\Z[i]$.
\end{proof}

A separate computation for the case $p = 13$, see Lemma~\ref{L:13noCM} below,
shows that Lemma~\ref{L:CM} remains valid in that case, even though the result
of~\cite{BPR2013} does not apply.

We can therefore further reduce the list of twists of~$X(p)$ that have to be
considered. This results in Table~\ref{Table-twist}, where an entry `$+$' (resp., `$-$')
indicates that the twist $X_E(p)$ (resp., $X^-_E(p)$) cannot (so far) be ruled out to
have rational points giving rise to a non-trivial primitive solution of~\eqref{E:GFE}.

\begin{table}[htb]
\[ \begin{array}{|r||c|c|c|c|c|c|c|} \hline
     p \bmod 24 & 27a1 & 54a1 & 96a1 & 288a1 & 864a1 & 864b1 & 864c1 \text{\large\strut}\\\hline
              1 &      &   +  &   +  &       &   +   &   +   &   +   \\
              5 &   +  &   -  &   +  &       &  + -  &  + -  &  + -  \\
              7 &      &   -  &   +  &   +   &   +   &   +   &   +   \\
             11 &   +  &   +  &   +  &  + -  &   +   &   +   &   +   \\
             13 &      &      &   -  &       &   +   &   +   &   +   \\
             17 &   +  &   +  &      &       &   +   &   +   &   +   \\
             19 &      &   +  &   -  &  + -  &  + -  &  + -  &  + -  \\
             23 &   +  &      &      &   +   &   +   &   +   &   +   \\\hline
   \end{array}
\]
\caption{Twists of~$X(p)$ remaining after local considerations and using
         information on $X_{\spl}^+(p)$, according to $p \bmod 24$.
         This table is valid for $p \ge 11$.}
\label{Table-twist}
\end{table}

Unfortunately, there is no similar result on mod~$p$ Galois representations whose image
is contained in the normalizer of a non-split Cartan subgroup. Such a result would allow
us to eliminate the curves $27a1$ and~$288a1$ also in the remaining cases.

In Section~\ref{SS:CM11} below we show how one can deal with the non-split
case when $p = 11$, by considering the twists $X^{(-1)}_\ns(p)$ and~$X^{(-3)}_\ns(p)$
of the double cover $X_\ns(p) \to X_\ns^+(p)$, where $X_\ns(p)$ classifies
elliptic curves such that the image of the mod~$p$ Galois representation
is contained in a non-split Cartan subgroup and $X_\ns^+(p)$ does the same
for the normalizer of a non-split Cartan subgroup. It turns out that
for $p = 11$ the two curves $X^{(-1)}_\ns(11)$ and~$X^{(-3)}_\ns(11)$ are
not directly amenable to a Chabauty argument; instead one can use suitable
coverings and Elliptic Curve Chabauty. The following argument shows that
the failure of the Chabauty condition is a general phenomenon.

By a result of Chen~\cite{Chen} (see also~\cite{Smit_Edixhoven}) the Jacobian
variety~$J_{0}(p^2)$ of~$X_{0}(p^2)$ is isogenous to the
product $\Jac(X_\ns(p))\times \Jac(X_{0}(p))^2$.
On the other hand, a theorem of Shimura~\cite{Shimura}*{Thm.~7.14} implies
that $J_{0}(p^2)$ is isogenous to the product $\prod_{f} A_{f}^{m_{f}}$,
where $f$ runs over a system of representatives of the Galois orbits of newforms
of  weight~$2$ and level~$M_{f}$ dividing~$p^2$,
$A_f$ is the abelian variety over~$\Q$ associated to~$f$ defined by Shimura,
and $p^{3-m_f} = M_{f}$.
It follows that $\Jac(X_\ns(p))$ is isogenous to the product of the~$A_{f}$
such that $f$ is a newform in~$S_{2}(\Gamma_{0}(p^2))$.
Similarly, the Jacobian of $X_{\ns}^+(p)$ corresponds to the product of the~$A_{f}$
for the subset of~$f$ invariant under the Atkin-Lehner involution~$W$ at level~$p^2$.

If $p \equiv -1 \bmod 4$, we need to exclude rational points on the twists $X_{288a1}^\pm(p)$;
solutions associated to this curve will give rise to rational points on the
$(-1)$-twist~$X_\ns^{(-1)}(p)$ of the double cover $X_\ns(p) \to X_\ns^{+}(p)$.
Similarly, for $p \equiv -1 \bmod 3$, we need to exclude rational points on the
twist $X_{27a1}(p)$, and solutions associated to that curve will give rise to rational
points on~$X_\ns^{(-3)}(p)$. To be able to use Chabauty's method, we would need
to have a factor of the Jacobian~$J_\ns^{(d)}(p)$ (for $d = -1$ and/or $d = -3$) of
Mordell-Weil rank strictly less than its dimension. Since all these factors
have real multiplication (defined over~$\Q$), the Mordell-Weil rank is always
a multiple of the dimension, so we actually need a factor of rank zero.

By the above, we know that $J_\ns^{(d)}(p)$ splits up to isogeny as the product
of the twists~$A_f^{(d)}$ for newforms~$f$ such that $f|_W = -f$ and the
untwisted~$A_f$ for $f$ such that $f|_W = f$. The $L$-series of~$A_f^{(d)}$
is the product of $L({}^\sigma \!f_\chi, s)$, where ${}^\sigma \!f$ runs through
the newforms in the Galois orbit and $\chi$ is the quadratic character
associated to~$d$, see~\cite{Shimura}*{Section 7.5}.
By a theorem of Weil~\cite{Weil1967}*{Satz~1} all these
$L$-series have root number~$-1$ when $f|_W = -f$ and $d < 0$ is squarefree
(note that $C = 1$ from $f|_W = -f$, $\eps$ is trivial, $\chi$ is real,
so $g(\chi) = g(\bar{\chi})$, and $A = p^2$, so that $\chi(-A) = \chi(-1) = -1$),
so $L(A_f^{(d)}, s)$ vanishes at least to
order~$\dim A_f^{(d)}$ at~$s = 1$. For the $f$ that are invariant under~$W$
we also have that the root number of~$L({}^\sigma\! f, s)$ is~$-1$, so
$L(A_f, s)$ also vanishes to order at least~$\dim A_f$.
Assuming the Birch and Swinnerton-Dyer conjecture, it follows that all factors
of~$J_\ns^{(d)}(p)$ have positive rank.

To conclude this section, we mention that when $p = 13$, we are in the split case
for both CM curves. During the work on this paper, it was an open question whether
the set of rational points on~$X^+_\spl(13)$ consists of cusps and CM~points.
(The curve is of genus~$3$ and its Jacobian has Mordell-Weil rank~$3$,
see~\cite{Baran2014} and~\cite{BPS}.) We tried an approach
similar to that used in Section~\ref{SS:CM11} below, but did not succeed.
However, a different approach using twists of~$X_1(13)$ is successful;
see Lemma~\ref{L:13noCM}.

After our work was finished, Steffen M\"uller announced at a workshop in Banff
that in joint work with Balakrishnan, Dogra, Tuitman and Vonk the set
$X^+_\ns(13)(\Q)$ could be determined using `Quadratic Chabauty'
techniques. This work has now appeared as~\cite{BDMTV}.
Since there is an accidental isomorphism $X^+_\ns(13) \simeq X^+_\spl(13)$
and there are no unexpected points, this gives another proof of Lemma~\ref{L:13noCM}.


\section{The Generalized Fermat Equation with exponents 2, 3, 11} \label{S:p=11}

We now consider the case $p = 11$.
In this section we will prove the following theorem.

\begin{theorem} \label{T:11}
  Assume the Generalized Riemann Hypothesis. Then the only primitive integral
  solutions of the equation $x^2 + y^3 = z^{11}$ are the trivial
  solutions $(\pm 1, 0, 1)$, $\pm(0, 1, 1)$, $(\pm 1, -1, 0)$ and the
  Catalan solutions $(\pm 3, -2, 1)$.
\end{theorem}

We note at this point that the Generalized Riemann Hypothesis is only
used to verify the correctness of the computation of the class groups
of five specific number fields of degree~$36$.

In the following we will say that $j \in \Q$ is \emph{good} if it is
the $j$-invariant of a Frey curve associated to a primitive integral
solution of $x^2 + y^3 = z^{11}$, which means that $j = (12b)^3/c^{11}$
and $12^3 - j = 12^3 a^2/c^{11}$ with coprime integers $a$, $b$, $c$.
In a similar way, we say that $j \in \Q_2$ is \emph{$2$-adically good}
if it has this form for coprime $2$-adic integers $a$, $b$, $c$.

By Theorem~\ref{T:nominus}, it suffices to find the rational points on the
twisted modular curves~$X_E(11)$ for the elliptic curves $E \in \calE'$, where
\[ \calE' = \{27a1, 54a1, 96a1, 288a1, 288a2, 864a1, 864b1, 864c1\}\,, \]
such that their image on the $j$-line is good.

Some of the results in this section rely on computations that require
a computer algebra system. We provide a script \texttt{section7.magma}
at~\cite{programs} (which relies on~\texttt{localtest.magma}, also provided there)
that can be loaded into \Magma\ and performs these computations.


\subsection{The CM curves} \label{SS:CM11} \strut

In the case $p = 11$,
we can deal with the CM curves $E \in \{27a1, 288a1, 288a2\}$ in the following way.
Note that since $(-1/11) = (-3/11) = -1$, the images of both relevant Galois
representations are contained in the normalizer of a non-split Cartan subgroup
of~$\GL_2(\F_{11})$. Elliptic curves with this property are parameterized by
the modular curve $X^+_\ns(11)$, which is the elliptic curve $121b1$ of rank~$1$.
It has as a double cover the curve $X_\ns(11)$ parameterizing elliptic curves~$E$
such that the image of the mod~$11$ Galois representation is contained in
a non-split Cartan subgroup. Elliptic curves whose mod~$11$ representation
is isomorphic to that of $288a1$ (or~$288a2$) or~$27a1$ will give rise to rational points
on the quadratic twists $X^{(-1)}_\ns(11)$ and~$X^{(-3)}_\ns(11)$ of this
double cover; see~\cite{DFGS2014}*{Remark 1}.
These curves are of genus~$4$; the Jacobian of~$X_\ns(11)$
is isogenous to the product of the four elliptic curves $121a1$, $121b1$,
$121c1$ and~$121d1$, so that the Jacobian of the twist $X^{(d)}_\ns(11)$
splits into the four elliptic curves $121b1$, $121a1^{(d)}$, $121c1^{(d)}$
and~$121d1^{(d)}$. Unfortunately, for $d = -1$ and~$d = -3$ all of these
curves have rank~$1$, so the obvious approach does not work. However, we
can use a covering collection combined with the Elliptic Curve Chabauty
method~\cite{Bruin2003}, as follows. An equation for $X^+_\ns(11)$ is
(see~\cite{Ligozat}*{Proposition~II.4.3.8.1})
\[ y^2 = 4 x^3 - 4 x^2 - 28 x + 41 \]
and the double cover $X_\ns(11) \to X^+_\ns(11)$ is given by
\[ t^2 = -(4 x^3 + 7 x^2 - 6 x + 19) \]
(this is an equation for $121c1$), see~\cite{DFGS2014}*{Proposition~1}.
Therefore our twists are given by
\[ X^{(-1)}_\ns(11) \colon \left\{ \begin{array}{r@{{}={}}l}
                                    y^2 & 4 x^3 - 4 x^2 - 28 x + 41 \\[2mm]
                                    t^2 & 4 x^3 + 7 x^2 - 6 x + 19
                                  \end{array} \right.
\]
and
\[ X^{(-3)}_\ns(11) \colon \left\{ \begin{array}{r@{{}={}}l}
                                    y^2 & 4 x^3 - 4 x^2 - 28 x + 41 \\[2mm]
                                    t^2 & 3(4 x^3 + 7 x^2 - 6 x + 19)\,.
                                   \end{array} \right.
\]
Let $\alpha$ be a root of $f_1(x) = 4 x^3 - 4 x^2 - 28 x + 41$ and set $K = \Q(\alpha)$.
Write $f_1(x) = (x - \alpha) g_1(x)$ in~$K[x]$.
Since $E_1 = 121b1$ has Mordell-Weil group~$E_1(\Q)$ isomorphic to~$\Z$, with generator
$P = (4,11)$, it follows that each rational point on~$E_1$ gives rise
to a $K$-rational point with rational $x$-coordinate on one of the two curves
\[ \left\{ \begin{array}{r@{{}={}}l}
              y_1^2 & x - \alpha \\[2mm]
              y_2^2 & g_1(x)
            \end{array} \right.
   \quad\text{and}\quad
   \left\{ \begin{array}{r@{{}={}}l}
              y_1^2 & (4 - \alpha) (x - \alpha) \\[2mm]
              y_2^2 & (4 - \alpha) g_1(x)\,.
            \end{array} \right.
\]
(Here we use that the map $E_1(\Q) \to K^\times/K^{\times 2}$ that associates
to a point~$P$ the square class of $x(P)-\alpha$ is a homomorphism.)
So a rational point on $X^{(d)}_\ns(11)$ will give a $K$-rational point with rational
$x$-coordinate on
\[ u^2 = -d (x - \alpha) (4 x^3 + 7 x^2 - 6 x + 19)
   \quad\text{or}\quad
   u^2 = -d (4 - \alpha) (x - \alpha) (4 x^3 + 7 x^2 - 6 x + 19) \,.
\]
These are elliptic curves over~$K$, which turn out to both have Mordell-Weil rank~$1$
for $d = -1$ and rank~$2$ for $d = -3$. Since the rank is strictly smaller than
the degree of~$K$ in all cases, Elliptic Curve Chabauty applies, and we find
using \Magma\ that the $x$-coordinates of the rational points on $X^{(-1)}_\ns(11)$
and~$X^{(-3)}_\ns(11)$ are $\infty, 5/4, 4, -2$, corresponding to $O$, $\pm 3P$,
$\pm P$ and $\pm 4P$ on~$E_1$.
We compute the $j$-invariants of the elliptic curves
represented by these points using the formula in~\cite{DFGS2014} and find that
only the curves corresponding to~$3P$ and to~$4P$ give rise to solutions
of~\eqref{E:GFE}; they are the trivial solutions with $a = 0$ or~$b = 0$.


\subsection{Dealing with the remaining curves} \strut

We now set $\calE = \{54a1, 96a1, 864a1, 864b1, 864c1\}$; this is the set
of curves~$E$ such that we still have to consider~$X_E(11)$.

We will denote any of the canonical morphisms
\[ X(11) \to X(1) \simeq \PP^1\,, \quad
   X_E(11) \simeq_{\bar{\Q}} X(11) \to X(1) \simeq \PP^1 \quad\text{and}\quad
   X_0(11) \to X(1) \simeq \PP^1
\]
by~$j$ and we will also use $j$ to denote the corresponding coordinate on~$\PP^1$.

Recall that $X_0(11)$ is an elliptic curve. Let $P \in X_E(11)(\Q)$ be
a rational point; then under the composition $X_E(11) \simeq X(11) \to X_0(11)$
(where the isomorphism is defined over~$\bar{\Q}$) $P$ will be mapped to
a point~$P'$ on~$X_0(11)$ whose image $j(P') = j(P)$ on the $j$-line is rational.
Since the $j$-map from~$X_0(11)$ has degree~$12$, it follows
that $P'$ is defined over a number field~$K$ of degree at most~$12$.
More precisely, the points in the fiber above $j(P') = j(P)$ in~$X_0(11)$
correspond to the twelve
possible cyclic subgroups of order~$11$ in~$E[11]$, so the Galois action
on the fiber depends only on~$E$ and is the same as the Galois action
on the fiber above the image~$j(E)$ on the $j$-line of the canonical point
of~$X_E(11)$. In particular, we can easily determine the isomorphism type
of this fiber. It turns out that for our five curves~$E$, the fiber is
irreducible, with a (geometric) point defined over a field $K  = K_E$ of degree~$12$.
The problem can therefore be reduced to the determination of the set
of $K_E$-points~$P'$ on~$X_0(11)$ such that $j(P') \in \Q$ and is good.
This kind of problem is the setting for the
Elliptic Curve Chabauty method as introduced in~\cite{Bruin2003}
that we have already used in Section~\ref{SS:CM11} above.
To apply the method, we need explicit generators of a finite-index
subgroup of the group~$X_0(11)(K_E)$. This requires knowing the rank
of this group, for which we can obtain an upper bound by computing a
suitable Selmer group. We use the $2$-Selmer group, whose computation
requires class and unit group information for the cubic extension~$L_E$
of~$K_E$ obtained by adjoining the $x$-coordinate of a point of
order~$2$ on~$X_0(11)$ (no field~$K_E$ has a non-trivial subfield,
so no point of order~$2$ on~$X_0(11)$ becomes rational over~$K_E$).
To make the relevant computation feasible, we assume the Generalized
Riemann Hypothesis. With this assumption the computation of the
$2$-Selmer groups is done by \Magma\ in reasonable time (up to
a few hours). However, we now have the problem that we do not find
sufficiently many independent points in~$X_0(11)(K_E)$ to reach the upper
bound. This is where an earlier attempt in~2006 along similar lines by
David Zureick-Brown got stuck. We get around this stumbling block by
making use of `Selmer Group Chabauty' as described in~\cite{Stoll2017b}.
This method allows us to work with the Selmer group information without
having to find sufficiently many points in~$X_0(11)(K_E)$.

The idea of the Selmer Group Chabauty method (when applied with the
$2$-Selmer group) is to combine the global information from the Selmer
group with local, here specifically $2$-adic, information. So we first
study our situation over~$\Q_2$. Away from the branch points
$0$, $12^3$ and~$\infty$ of $j \colon X_0(11) \to \PP^1_j$, the
$\Q_2$-isomorphism type of the fiber is locally constant in the $2$-adic
topology. In a suitable neighborhood of a branch point, the isomorphism
type of the fiber will only depend on the class of the value of a
suitable uniformizer on~$\PP^1_j$ at the branch point modulo cubes
(for~$0$), squares (for~$12^3$) or eleventh powers (for~$\infty$).
We use the standard model given by
\[ y^2 + y = x^3 - x^2 - 10 x - 20 \]
for the elliptic curve~$X_0(11)$, with $j$-invariant map given by
$j = (a(x) + b(x) y)/(x - 16)^{11}$, where
\begin{align*}
  a(x) &= 743 x^{11} + 21559874 x^{10} + 19162005343 x^9 + 2536749758583 x^8 \\
       &\quad{} + 82165362766027 x^7 + 576036867160006 x^6 - 1895608370650736 x^5 \\
       &\quad{} - 14545268641576841 x^4 + 420015065507429 x^3 + 74593328129816300 x^2 \\
       &\quad{} + 108160113602504237 x - 39176677684144739 \\
  \intertext{and}
  b(x) &= (x^5 + 4518 x^4 + 1304157 x^3 + 65058492 x^2 + 271927184 x - 707351591) \\
       &\qquad{} \cdot (x^5 + 192189 x^4 + 3626752 x^3 - 3406817 x^2 - 37789861 x - 37315543) \,.
\end{align*}

We define the following set of subsets of~$\PP^1(\Q_2)$.
\begin{align}
   \calD = \bigl\{& 15 \cdot 2^6 + 2^{11} \Z_2, \; - 2^6 + 2^{11} \Z_2, \;
                    2^9 + 2^{11} \Z_2, \; -2^9 + 2^{11} \Z_2, \;
                    \{2^{-5} t^{-11} : t \in \Z_2\}, \nonumber \\
                  & \{12^3 - 3 \cdot 2^{10} t^2 : t \in \Z_2\}, \;
                    \{12^3 - 2^{10} t^2 : t \in \Z_2\}, \label{E:defD} \\
                  & \{12^3 + 2^{10} t^2 : t \in \Z_2\}, \;
                    \{12^3 + 3 \cdot 2^{10} t^2 : t \in \Z_2\}\bigr\} \nonumber
\end{align}
Note that according to Table~\ref{Tab:2adic} all elements in these sets
are $2$-adically good $j$-invariants, and for each set, all fibers of
the $j$-map $X_0(11) \to \PP^1$ over points in the set are isomorphic over~$\Q_2$
(excluding $j = 12^3$ and $j = \infty$).

\begin{lemma} \label{L:list1}
  Let $E \in \calE$ and let $P \in X_E(11)(\Q_2)$ such that $j(P)$ is $2$-adically good.
  Then $j(P)$ is in one of the following sets $D \in \calD$, depending on~$E$.
  \begin{align*}
     54a1 \colon & \{2^{-5} t^{-11} : t \in \Z_2\} \,. \\
     96a1 \colon & 15 \cdot 2^6 + 2^{11} \Z_2, \; {-2^6} + 2^{11} \Z_2, \;
                   {-2^9} + 2^{11} \Z_2 \,. \\
    864a1 \colon & \{12^3 - 2^{10} t^2 : t \in \Z_2\}, \;
                   \{12^3 + 3 \cdot 2^{10} t^2 : t \in \Z_2\} \,. \\
    864b1 \colon & \{12^3 - 3 \cdot 2^{10} t^2 : t \in \Z_2\}, \; \{12^3 + 2^{10} t^2 : t \in \Z_2\}, \;
                   2^9 + 2^{11} \Z_2 \,. \\
    864c1 \colon & 15 \cdot 2^6 + 2^{11} \Z_2, \; -2^6 + 2^{11} \Z_2, \; {-2^9} + 2^{11} \Z_2 \,.
  \end{align*}
\end{lemma}

\begin{proof}
  This follows from the information in Table~\ref{Tab:2adic}, together with
  Lemma~\ref{L:fine2}, which allows us to distinguish between $X_E(p)$ and~$X_E^-(p)$
  (note that $2$ is a non-square mod~$11$).
\end{proof}

The next step is the computation of the $2$-Selmer groups of~$X_0(11)$
over the fields~$K_E$, where $E$ runs through the curves in~$\calE$.
This is where we assume GRH. Table~\ref{Table:KE} lists defining polynomials
for the fields~$K_E$ and gives the $\F_2$-dimension of the Selmer group.

\begin{table}[htb]
\renewcommand{\arraystretch}{1.25}
\begin{tabular}{|c|c|c|} \hline
  $E$ & polynomial defining $K_E$ & $\dim_{\F_2} \Sel_2/K_E$ \\ \hline
  $54a1$  & $x^{12} - 6 x^{10} + 6 x^9 - 6 x^8 - 126 x^7 + 104 x^6 + 468 x^5$ \hfill\strut & 4 \\[-3pt]
          & \strut\hfill${} + 258 x^4 - 456 x^3 - 1062 x^2 - 774 x - 380$ & \\ \hline
  $96a1$  & $x^{12} - 4 x^{11} - 264 x^7 + 66 x^6 - 132 x^5$ \hfill\strut & 5 \\[-3pt]
          & \strut\hfill${} - 2112 x^4 - 1320 x^3 - 660 x^2 - 6240 x - 8007$ & \\ \hline
  $864a1$ & $x^{12} - 6 x^{11} + 110 x^9 - 132 x^8 - 528 x^7 + 1100 x^6 + 330 x^5$ \hfill\strut & 5 \\[-3pt]
          & \strut\hfill${} - 2508 x^4 + 2134 x^3 - 594 x^2 + 456 x - 371$ & \\ \hline
  $864b1$ & $x^{12} - 6 x^{11} + 22 x^9 + 99 x^8 - 396 x^7 + 440 x^6 - 132 x^5$ \hfill\strut & 3 \\[-3pt]
          & \strut\hfill${} - 6501 x^4 + 33506 x^3 - 23760 x^2 - 92418 x + 193081$ & \\ \hline
  $864c1$ & $x^{12} - 44 x^9 - 264 x^8 - 264 x^7 - 2266 x^6 - 4488 x^5$ \hfill\strut & 3 \\[-3pt]
          & \strut\hfill ${} - 264 x^4 - 17644 x^3 - 7128 x^2 + 144 x - 15191$ & \\ \hline
\end{tabular}

\medskip

\caption{Fields $K_E$ and dimensions of Selmer groups, for $E \in \calE$.}
\label{Table:KE}
\end{table}

We eliminate~$y$ from the equation of~$X_0(11)$ and the relation between~$j$
and~$x,y$. This results in
\begin{align}
  F(x,j) =
  (x^4 &- 52820 x^3 + 1333262 x^2 + 4971236 x + 9789217)^3 \nonumber \\
    &{} + (1486 x^{11} + 43119747 x^{10} + 38323813979 x^9 + 5072626276355 x^8 \nonumber \\
    &\qquad{} + 164063633585170 x^7 + 1134855511654843 x^6 - 4074814667347831 x^5  \label{E:Fxj} \\
    &\qquad{} - 29669709666741936 x^4 + 6839041777752481 x^3 + 159480622275659333 x^2 \nonumber \\
    &\qquad{} + 199736619430410535 x - 104748564078368391) j \nonumber \\
    &{} - (x - 16)^{11} j^2
    = 0 \,. \nonumber
\end{align}

We now state a technical lemma for later use.

\begin{lemma} \label{L:tech}
  Let $K$ be a field complete with respect to an absolute value~$|{\cdot}|$.
  Consider a polynomial $F = \sum_{i,j \ge 0} f_{ij} x^i y^j \in K[x,y]$.
  Fix an integer $e \ge 1$ such that the characteristic of~$K$ does not divide~$e$.
  We assume that $f_{0j} = 0$ for $0 \le j < e$ and that $f_{0e} = 1$.
  For $j \ge 0$, we set $F_j(t) = \sum_{i \ge 0} f_{ij} t^i \in K[t]$, so that
  $F(x,y) = \sum_{j \ge 0} F_j(x) y^j$. Assume that $F_0(t) = -c t + \text{higher order terms}$,
  with $c \neq 0$.
  For a real number $r > 0$ and a polynomial
  $f(t) = a_n t^n + \ldots + a_1 t + a_0 \in K[t]$, we set $|f|_r = \max\{r^i |a_i| : 0 \le i \le n\}$.

  There are exactly $e$ formal power series $\phi_0, \ldots, \phi_{e-1} \in L\pws{t}$,
  where $L$ is the splitting field of $X^e - c$ over~$K$, such that $\phi_j(0) = 0$
  and $F(t^e, \phi_j(t)) = 0$. If $\zeta \in L$ is a primitive $e$th root of unity,
  then we can label the~$\phi_j$ in such a way that $\phi_j(t) = \phi_0(\zeta^j t)$.

  If $r > 0$ is such that $|F_m(t^e)|_r < |F_0(t^e)|_r^{(e-m)/e}$ for $0 < m < e$,
  $|F_e(t^e) - 1|_r < 1$ and $|F_0(t^e)|_r^{(m-e)/e} |F_m(t^e)|_r < 1$ for $m > e$,
  then the $\phi_j$ converge on the closed disk of radius~$r$ in~$L$,
  and we have $|\phi_j(\tau)| \le |F_0(t^e)|_r^{1/e}$ for all $\tau$ in this disk.
  If in addition $|F_0(t^e) + c t^e|_r < |F_0(t^e)|_r$, then $|\phi_j(\tau)| = |c|^{1/e} |\tau|$
  for all these~$\tau$.
\end{lemma}

\begin{proof}
  We consider the equation $F(t^e, y) = 0$ over the field of Laurent series $L\lrs{t}$.
  The assumptions $f_{0j} = 0$ for $0 \le j < e$, $f_{0e} = 1$ and $f_{10} = -c \neq 0$
  imply that the Newton Polygon of~$F(t^e, y)$ has a segment of length~$e$
  and slope~$-e$, and that all other slopes are~$\ge 0$. This already shows
  that there are at most~$e$ power series with the required properties.
  Also, the reduction modulo~$t$ of $t^{-e} F(t^e, tz)$ is $z^e - c$, which is a separable
  polynomial splitting over~$L$ into linear factors. So by Hensel's Lemma,
  there are exactly~$e$ solutions $z \in L\pws{t}$. Let $\phi_0 = t z_0$ for
  one such solution~$z_0$. Clearly, each solution~$z$ gives rise to exactly one
  power series~$\phi_j$. Since the original equation is invariant under the substitution
  $t \mapsto \zeta^j t$, $\phi_j(t) := \phi_0(\zeta^j t)$ is a solution for each~$0 \le j < e$.
  Since $\gamma = z_0(0) \neq 0$ (it is an $e$th root of~$c$), we have $\phi_0(t) = \gamma t + \ldots$,
  and so all these~$\phi_j$ are pairwise distinct.

  Now consider the completion of the polynomial ring~$L[t]$ with respect to~$|{\cdot}|_r$.
  This is the Tate Algebra~$T_r$ of power series converging on the closed disk of
  radius~$r$ (in the algebraic closure of~$L$). The assumptions on~$r$ guarantee
  that the Newton Polygon of~$F(t^e, y)$, considered over~$T_r$, again has a
  unique segment of length~$e$ and slope corresponding to the absolute value~$|F_0(t^e)|_r^{1/e}$,
  whereas all other slopes correspond to larger absolute values. As can be seen by
  letting $r$ tend to zero, the corresponding solutions must be given by the~$\phi_j$.
  The claim that $|\phi_j(\tau)| \le |F_0(t^e)|_r^{1/e}$ follows from $|\phi_j|_r = |F_0(t^e)|_r^{1/e}$.
  For the last claim, note that $|F_0(t^e)|_r = |c| r^e$ and that if $|\phi_j(\tau)| < |c|^{1/e}$,
  then the term $-c \tau^e$ would be dominant in~$F(\tau^e, \phi_j(\tau))$, which
  gives a contradiction.
\end{proof}

Recall that $\theta$ denotes the $x$-coordinate of a point of order~$2$
on~$X_0(11)$, so $\theta$ is a root of the $2$-division polynomial
\[ 4 x^3 - 4 x^2 - 40 x - 79 \]
of~$X_0(11)$. We denote the $2$-adic valuation on~$\bar{\Q}_2$ by~$v_2$,
normalized so that $v_2(2) = 1$. Then $v_2(\theta) = -2/3$.

\begin{lemma} \label{L:const}
  Let $K$ be a finite extension of~$\Q_2$ such that $X_0(11)(K)[2] = 0$
  and set $L = K(\theta)$. Let $D \in \calD$, but different from~$\{2^{-5} t^{-11} : t \in \Z_2\}$,
  and let $\varphi \colon \Z_2 \to D$
  be the parameterization in terms of~$t$ as given in~\eqref{E:defD}. Let $P \in X_0(11)(K)$
  be such that $j(P) \in D$. Then $P$ is in the image of an analytic
  map $\phi \colon \Z_2 \to X_0(11)(K)$ such that $j \circ \phi = \varphi$,
  and the square class of $x(\phi(z)) - \theta \in L^\times$ is constant
  for $z \in \Z_2$.
\end{lemma}

\begin{proof}
  By the information in Table~\ref{Tab:2adic} the $\Q_2$-isomorphism type of the fiber of~$j$
  above~$D$ is constant, say given by the disjoint union of $\Spec K_i$
  for certain $2$-adic fields~$K_i$.
  We note that all~$K_i$ coming up in this way have the property that
  $X_0(11)(K_j)[2] = 0$, since the ramification indices are not divisible by~$3$.
  We can then take $K = K_i$, since we will use a purely valuation-theoretic criterion
  for the second statement, so the choice of field will be largely irrelevant.
  In the following, $F(x,j) = 0$ is the relation between the $x$-coordinate on~$X_0(11)$
  and the associated $j$-invariant given in~\eqref{E:Fxj}.

  If $D$ is not one of the last four sets in~\eqref{E:defD},
  then we solve $F(x_0 + \phi_0, \varphi(t)) = 0$ for a power series~$\phi_0$
  with $\phi_0(0) = 0$, where $x_0 \in K$ is any root of~$F(x, \varphi(0))$.
  For each possible $D$ and~$x_0$,
  Lemma~\ref{L:tech} allows us to deduce that such a power series~$\phi_0$
  exists, that it converges on an open disk containing~$\Z_2$ and that it
  satisfies $v_2(\phi_0(\tau)) > 4/3$ for all $\tau \in \Z_2$.
  Since we obtain as many power series as there are roots of $F(x, z) = 0$ for any $z \in D$
  (this is because the fibers over~$D$ of the $j$-map are all isomorphic,
  so the number of roots in~$K$ is constant)
  and since the $y$-coordinate of a point on~$X_0(11)$ is uniquely determined
  by its $x$-coordinate and its image under the $j$-map (unless $b(x) = 0$, but
  this never happens for the points we are considering), we obtain analytic
  maps $\phi \colon \Z_2 \to X_0(11)(K)$ whose images cover all points~$P$ as in
  the statement. Also, $v_2(x(\phi(\tau)) - x_0) = v_2(\phi_0(\tau)) > 4/3 = 2 + v_2(x_0 - \theta)$,
  which implies (compare~\cite{Stoll2001}*{Lemma~6.3}) that $x(\phi(\tau)) - \theta$
  is in the same square class as~$x_0 - \theta$.

  If $D$ is one of the last four sets in~\eqref{E:defD}, then we proceed in
  a similar way. Note that the four parameterizations $\varphi(t) = 12^3 + a 2^{10} t^2$
  all have $2$-adic units~$a$ and so can be converted one into another by
  scaling~$t$ by a $2$-adic unit (in some extension field of~$\Q_2$).
  Since we only care about valuations in the argument, it is sufficient
  to just work with one of them. In the same way as before, we consider
  $F(x_0 + \phi_0(t), \varphi(t)) = 0$ as an equation to be solved
  for a power series~$\phi_0$ with $\phi_0(0) = 0$, where $x_0 \in K$
  is such that $F(x_0, 12^3) = 0$. We apply Lemma~\ref{L:tech} again,
  this time with $e = 2$, and the remaining argument is similar to the previous case.
\end{proof}

We now use the information coming from the Selmer group together with the
preceding lemma to rule out most of the sets listed in Lemma~\ref{L:list1}.

\begin{lemma} \label{L:list2}
  Let $E \in \calE$ and let $P \in X_E(11)(\Q)$ such that $j(P)$ is $2$-adically good.
  Then $j(P)$ is in one of the following sets $D \in \calD$, depending on~$E$.
  \begin{align*}
     54a1 &\colon \{2^{-5} t^{-11} : t \in \Z_2\} \,. \\
     96a1 &\colon 15 \cdot 2^6 + 2^{11} \Z_2, \; {-2^6} + 2^{11} \Z_2 \,. \\
    864a1 &\colon \text{\rm none}\,. \\
    864b1 &\colon 2^9 + 2^{11} \Z_2 \,. \\
    864c1 &\colon {-2^9} + 2^{11} \Z_2 \,.
  \end{align*}
\end{lemma}

Note that the curve $864a1$ can already be ruled out at this stage.

\begin{proof}
  In view of Lemma~\ref{L:list1}, there is nothing to prove when $E = 54a1$.
  So we let $E$ be one of the other four curves.
  Any rational point on~$X_E(11)$ whose image on the $j$-line is good
  will map to a point
  in $X_0(11)(K_E)$ with the same $j$-invariant, and so will give rise to
  a point in~$X_0(11)(K_E \otimes_{\Q} \Q_2)$ whose $j$-invariant is in one
  of the sets~$D$ listed in Lemma~\ref{L:list1}, depending on~$E$.
  Recall that $L_E = K_E(\theta)$.
  We write $K_{E,2} = K_E \otimes_{\Q} \Q_2$ and $L_{E,2} = L_E \otimes_{\Q} \Q_2$;
  $K_{E,2}$ and~$L_{E,2}$ are \'etale algebras over~$\Q_2$.
  Then we have the commutative diagram
  \[ \xymatrix{ X_0(11)(K_E) \ar[r] \ar[d]
                  & \Sel_2(X_0(11)/K_E) \ar@{^(->}[r] \ar[d]
                  & \dfrac{L_E^\times}{L_E^{\times 2}} \ar[d] \\
                X_0(11)(K_{E,2}) \ar[r]
                  & \dfrac{X_0(11)(K_{E,2})}{2 X_0(11)(K_{E,2})} \ar@{^(->}[r]
                  & \dfrac{L_{E,2}^\times}{L_{E,2}^{\times 2}}\,.
              }
  \]
  The composition of the two horizontal maps in the bottom row sends a
  point $(\xi,\eta)$ to the square class of $\xi - \theta$ in~$L_{E,2}^\times$.
  By Lemma~\ref{L:const}, the square class we obtain for a point
  in~$X_0(11)(K_{E,2})$ mapping into a fixed set~$D$ does not depend on the image point in~$D$.
  It therefore suffices to compute the square class for the points above
  some representative point (for example, the `center' if it is not a branch point)
  of~$D$. Doing this, we find that the square classes we obtain
  are not in the image of the Selmer group except for the sets given in
  the statement. Since by the diagram above a point in~$X_0(11)(K_E)$ has to map
  into the image of the Selmer group, this allows us to exclude these~$D$.
\end{proof}

It remains to deal with the remaining five sets~$D$. All but one of them
do actually contain the image of a point in~$X_0(11)(K_E)$, so we have to
use a more sophisticated approach.
The idea for the following statement comes from~\cite{Stoll2017b}.

\begin{lemma}
  Let $E \in \calE$ and let $D \in \calD$ be one of the sets associated
  to~$E$ in Lemma~\ref{L:list2}. Assume that there is a point
  $P \in X_0(11)(K_E)$ with the following property.

  \begin{itemize}
    \item[(*)] For any point $Q \in X_0(11)(K_{E,2})$ with $Q \neq P$ and $j(Q) \in D$,
               there is $n \ge 0$ such that $Q = P + 2^n Q'$
               with $Q' \in X_0(11)(K_{E,2})$ such that the image of~$Q'$
               in~$L_{E,2}^\times/L_{E,2}^{\times 2}$ is not in the image
               of the Selmer group.
  \end{itemize}

  Then if $j(P) \in D$, $P$ is the only point $Q \in X_0(11)(K_E)$ with $j(Q) \in D$,
  and if $j(P) \notin D$, then there is no such point.
\end{lemma}

\begin{proof}
  For each~$E \in \calE$, we verify that the middle vertical map
  in the diagram in the proof of Lemma~\ref{L:list2} is injective, by checking
  that the rightmost vertical map is injective on the image of the Selmer group.
  Note that the Selmer group is actually computed as a subgroup of the upper right
  group. Since $X_0(11)(K_E)/2 X_0(11)(K_E)$ maps injectively into the Selmer
  group, this means that a $K_E$-rational point that is divisible by~$2$
  in~$X_0(11)(K_{E,2})$ is already divisible by~$2$ in~$X_0(11)(K_E)$. Since $X_0(11)$
  has no $K_E$-rational points of exact order~$2$ (none of the fields~$K_E$
  have non-trivial subfields, so $\theta \notin K_E$, since $[\Q(\theta) : \Q] = 3)$,
  there is a unique `half' of a point, if there is any.
  So if $P \neq Q \in X_0(11)(K_E)$ has $j$-invariant
  in~$D$, then the point~$Q'$ in the relation in property~(*)
  is also $K_E$-rational. But then its image in~$L_{E,2}^\times/L_{E,2}^{\times 2}$
  must be in the image of the Selmer group, which gives a contradiction to~(*).
  The only remaining possibility for a point~$Q \in X_0(11)(K_E)$ with
  $j(Q) \in D$ is then~$P$, and this possibility only exists when $j(P) \in D$.
\end{proof}

It remains to exhibit a suitable point~$P$ for the remaining pairs $(E,D)$
and to show that it has property~(*). We first have a look at the $2$-adic
elliptic logarithm on~$X_0(11)$. Let $\calK \subset X_0(11)(\bar{\Q}_2)$ denote
the kernel of reduction. We take $t = -x/y$ to be a uniformizer at the point
at infinity on~$X_0(11)$ and write $\calK_\nu = \{P \in \calK : v_2(t(P)) > \nu\}$.

\begin{lemma} \label{L:logiso}
  The $2$-adic elliptic logarithm~$\log \colon \calK \to \bar{\Q}_2$
  induces a group isomorphism between~$\calK_{1/3}$
  and the additive group $D_{1/3} = \{\lambda \in \bar{\Q}_2 : v_2(\lambda) > 1/3\}$.

  In particular, if $K$ is a $2$-adic field and $P \in \calK_{4/3} \cap X_0(11)(K)$,
  then $P$ is divisible by~$2$ in~$\calK \cap X_0(11)(K)$.
\end{lemma}

\begin{proof}
  Note that the points~$T$ of order~$2$ on~$X_0(11)$ satisfy $v_2(t(T)) = 1/3$
  (the $x$-coordinate has valuation $-2/3$ and the $y$-coordinate is~$-1/2$).
  We also note that $X_0(11)$ is supersingular at~$2$, so $X_0(11)(\bar{\F}_2)$
  consists of points of odd order. This implies that the kernel of
  $\log$ on~$\calK$ consists exactly of the points of order a power of~$2$.
  There are no such points~$P$ with $v_2(t(P)) > 1/3$, so $\log$ is injective
  on this set. Explicitly, we find that for $P \in \calK$ with $t(P) = \tau$,
  \[ \log P =  \tau - \frac{1}{3} \tau^3 + \frac{1}{2} \tau^4 - \frac{19}{5} \tau^5
                 - \tau^6 + \frac{5}{7} \tau^7 - \frac{27}{2} \tau^8 + \ldots \,;
  \]
  for $v_2(\tau) > 1/3$ the first term is dominant, so the image is~$D_{1/3}$
  as claimed.

  Now let $P \in \calK_{4/3} \cap X_0(11)(K)$. Note that restricting~$\log$ gives
  us an isomorphism between $\calK_{1/3} \cap X_0(11)(K)$ and~$D_{1/3} \cap K$.
  Since $v_2(\tau) > 4/3$, the image of~$P$ in $D_{1/3} \cap K$
  is divisible by~$2$ in~$D_{1/3} \cap K$,
  so $P$ must be divisible by~$2$ in~$\calK \cap X_0(11)(K)$.
\end{proof}

\begin{lemma} \label{L:qconst}
  Let $K$ be some $2$-adic field such that $X_0(11)(K)[2] = 0$ and consider an analytic map
  $\phi \colon \Z_2 \to \calK \cap X_0(11)(K)$.
  We assume that $\phi$ actually converges on the open disk
  $D_{-c} = \{\tau \in \C_2 : v_2(\tau) > -c\}$ for some $c > 0$
  and that there is some $\nu \in \Q_{> 1/3}$ such that
  $|t(\phi(\tau))| = |2|^\nu\,|\tau|$ for all $\tau \in D_{-c}$.
  (This implies that $\phi(0)$ is the point at infinity.)
  We set $\mu = \lceil \nu - \tfrac{1}{3} \rceil - 1 \in \Z_{\ge 0}$.

  Then for each $\tau \in \Z_2 \setminus \{0\}$ there is
  a unique $Q_\tau \in \calK \cap X_0(11)(K)$ such that
  $\phi(\tau) = 2^{\mu+v_2(\tau)} Q_\tau$.
  If $\tau \in \Z_2^\times$, then $Q_\tau \equiv Q_1 \bmod 2 X_0(11)(K)$.
  If $\nu-\mu + \min\{1,c,\nu-\tfrac{1}{3}\} > \tfrac{4}{3}$, then this
  remains true for arbitrary $\tau \in \Z_2 \setminus \{0\}$. Otherwise,
  we have that $Q_\tau \equiv Q_2 \bmod 2 X_0(11)(K)$ if $\tau \in 2\Z_2 \setminus \{0\}$.
\end{lemma}

\begin{proof}
  Fix $0 \neq \tau \in \Z_2$ and write $n = v_2(\tau)$.
  By assumption, we have that $\phi(\tau) \in \calK_{\nu+n}$.
  Since $\nu > 1/3+\mu$, this implies that $\phi(\tau)$ is divisible by~$2^{\mu+n}$
  in~$\calK \cap X_0(11)(K)$ by Lemma~\ref{L:logiso}.
  Since $X_0(11)(K)[2] = 0$, there is then a unique point
  $Q_\tau \in \calK \cap X_0(11)(K)$ such that $\phi(\tau) = 2^{\mu+n} Q_\tau$.
  We now consider
  \[ \log \phi(\tau) = \gamma \bigl(\tau + a_2 \tau^2 + a_3 \tau^3 + \ldots\bigr) \,. \]
  We know by assumption and by Lemma~\ref{L:logiso} and its proof that
  $|\log \phi(\tau)| = |t(\phi(\tau))| = |2|^\nu\,|\tau|$ whenever
  $v_2(\tau) > -\min\{c, \nu-\tfrac{1}{3}\}$ (for $\tau \in \C_2$). This implies that
  $v_2(\gamma) = \nu$ and that
  \[ v_2(a_k) \ge (k-1) \min\{c, \nu-\tfrac{1}{3}\} \qquad \text{for all $k \ge 2$.} \]
  Writing $\tau = 2^n u$ with $u \in \Z_2^\times$, we then have that
  \[ \log Q_\tau = 2^{-\mu-n} \log \phi(\tau)
                 = \gamma 2^{-\mu} (u + 2^{n} a_2 u^2 + 2^{2n} a_3 u^3 + \ldots)
  \]
  and so
  \[ \log (Q_\tau-Q_1) = \log Q_\tau - \log Q_1
                       = \gamma 2^{-\mu} \bigl((u-1) + a_2 (2^n u^2-1) + a_3 (2^{2n}u^3-1) + \ldots\bigr) \,.
  \]
  If $n = 0$, then $v_2(\log(Q_\tau-Q_1)) \ge \nu-\mu+1 > \tfrac{4}{3}$, and if
  $\nu-\mu + \min\{1,c,\nu-\tfrac{1}{3}\} > \tfrac{4}{3}$, then
  $v_2(\log(Q_\tau-Q_1)) \ge \nu-\mu+\min\{1,v_2(a_2),v_2(a_3),\ldots\} > \tfrac{4}{3}$,
  so in both cases, Lemma~\ref{L:logiso} shows that $Q_\tau-Q_1$ is divisible by~$2$
  in~$\calK \cap X_0(11)(K)$. If $n \ge 1$, then we find that
  \[ \log (Q_\tau-Q_2) = \gamma 2^{-\mu} \bigl((u-1) + 2 a_2 (2^{n-1} u^2 - 1) + 2^2 a_3 (2^{2n-2} u^3 - 1)
                                                 + \ldots \bigr)\,,
  \]
  and we see that this has $2$-adic valuation $> \tfrac{4}{3}$, so $Q_\tau - Q_2$
  is divisible by~$2$.
\end{proof}

\begin{remark} \label{R:checkstar}
  We will apply this lemma in the following setting. We consider one of the
  remaining sets $D \in \calD$ and a point $P \in X_0(11)(K_E)$ such that $j(P) \in D$.
  Then there is an analytic map $\psi \colon \Z_2 \stackrel{\simeq}{\to} D \to X_0(11)(K_{E,2})$
  such that $\psi(0) = P$ and such that the second map in this composition
  inverts the $j$~function. We set $\phi(\tau) = \psi(\tau) - P$; then $\phi$
  is an analytic map into $\calK \cap X_0(11)(K_{E,2})$, which satisfies
  a polynomial relation $\Phi(\tau, t(\phi(\tau))) = 0$, where $\Phi \in K_E[x,y]$
  has degree~$3$ in~$x$ and degree~$12$ in~$y$ (this is because $j \colon X_0(11) \to \PP^1$
  has degree~$12$ and $t \colon X_0(11) \to \PP^1$ has degree~$3$).
  We find $\Phi$ explicitly by interpolation: we compute the pairs $(j(Q), t(Q))$
  for all points $Q = nP + T$, where $-5 \le n \le 5$ and $T \in X_0(11)(\Q)_{\tors} \simeq \Z/5\Z$.
  This gives us enough information to determine the $52$~coefficients of~$\Phi$
  (up to scaling). We then use Lemma~\ref{L:tech} to show that $\phi$ actually
  converges for $v_2(\tau) > -c$ for some $c > 1/3$ and that $v_2(t(\phi(\tau))) = \nu + v_2(\tau)$
  for $v_2(\tau) > -c$ (where $\nu = 2$, $3$ or~$5/11$ in the concrete cases considered).
  Finally, we check that $Q_1$ does not map into the image of the Selmer group
  (and that $Q_2$ does not, either, in the last case).
  This then verifies condition~(*) for $D$, $E$ and~$P$.
\end{remark}

\begin{lemma}
  Let $E \in \calE$ and let $D \in \calD$ be one of the sets associated
  to~$E$ in Lemma~\ref{L:list2}. Then the point~$P$ given in the table
  below satisfies~(*) for $E$ and~$D$. Here $\can(E)$ stands for the
  image on~$X_0(11)$ of the canonical point on~$X_E(11)$.
  \[ \renewcommand{\arraystretch}{1.2}
     \begin{array}{|c|c|c|c|} \hline
       E & D & P & j(P) \in D \setminus \{0, 12^3, \infty\} \\\hline
        54a1 & \{2^{-5} t^{-11} : t \in \Z_2\} & (16, 60)      & \text{\rm no} \\
        96a1 & 15 \cdot 2^6 + 2^{11} \Z_2      & {-\can(96a2)} & \text{\rm no} \\
        96a1 & {-2^6} + 2^{11} \Z_2            & \can(96a1)    & \text{\rm yes} \\
       864b1 & 2^9 + 2^{11} \Z_2               & \can(864b1)   & \text{\rm yes} \\
       864c1 & {-2^9} + 2^{11} \Z_2            & \can(864c1)   & \text{\rm yes} \\\hline
     \end{array}
  \]
\end{lemma}

\begin{proof}
  The points~$P$ given in the table have the property that their image
  in~$G = L_{E,2}^\times/L_{E,2}^{\times 2}$ agrees with the image of those
  points in $j^{-1}(D) \cap X_0(11)(K_{E,2})$ whose image is in the image
  of the Selmer group. This means that for any point~$Q$ in one of the
  $2$-adic disks above~$D$ such that $Q$ maps into the image of the Selmer
  group, we have that $P - Q$ is divisible by~$2$ in~$X_0(11)(K_{E,2})$.

  We first consider the last three cases.
  In the last two cases only one
  (out of sixteen) of the disks above~$D$ maps into the image of the
  Selmer group; this must be the disk containing the image of the canonical
  point. For $96a1$ we use Fisher's explicit description of~$X_E(11)$
  and the $j$-map on it~\cite{Fisher2014} to obtain a partition of~$X_{96a1}(11)(\Q_2)$
  into $2$-adic disks;
  we find that there is only one residue disk in~$X_{96a1}(11)(\Q_2)$
  that maps to~$D$. Its image in~$X_0(11)(K_{96a1,2})$ must be one of the disks
  above~$D$ and this disk contains the image of the canonical point; the other
  disks above~$D$ can be excluded. So in each of these three cases
  the only disk in~$X_0(11)(K_{E,2})$ above~$D$ that we have to consider
  is the disk~$D'$ containing the image of the canonical point.
  We then follow the approach outlined in Remark~\ref{R:checkstar} above
  to verify~(*) for each of the last three entries in the table.

  Next we consider the other disk~$D$ for $E = 96a1$; this is the second entry
  in the table. There are four disks~$D'$ in~$X_0(11)(K_{E,2})$
  above~$D$ such that the image of~$D'$ in~$G$ is in the image of the Selmer group;
  this image is the same as that of~$P$. Taking the difference with~$P$
  and halving, we find that on three of these disks the image in~$G$ of
  the resulting points is not in the image of the Selmer group.
  On the fourth disk, the image is zero, so the points are again divisible
  by~$2$. After halving again, we find that the resulting points have
  image in~$G$ not in the image of the Selmer group. This verifies~(*)
  for this case (with $n \le 2$).

  Finally, we look at~$E = 54a1$. There is one unramified branch above $j = \infty$
  with the point at infinity of~$X_0(11)$ sitting in the center of the disk,
  and there is one point (with coordinates $(16, 60)$) with ramification index~$11$.
  We can parameterize
  the disk relevant to us by setting $j = 2^{-5} \tau^{-11}$ and solving for
  the $x$ and $y$-coordinates in~$\Q(\!\sqrt[11]{2})(\!(\tau)\!)$.
  We use the alternative uniformizer $t' = -(x-5)/(y-5)$ at the origin
  (the standard uniformizer~$t$ does not work well in this case, because
  it is $2$-adically small on the other branch above $j = \infty$),
  which has the additional benefit that it is of degree~$2$ instead of~$3$ as
  a function on~$X_0(11)$, leading to a smaller polynomial~$\Phi$.
  The statements of Lemmas~\ref{L:logiso} and~\ref{L:qconst} are unaffected
  by this change of uniformizer. Using the approach of Remark~\ref{R:checkstar},
  we find that the series~$\phi$ of Lemma~\ref{L:qconst} converges for
  $v_2(\tau) >  -5/11$ and satisfies $v_2(\phi(\tau)) = 5/11 + v_2(\tau)$.
  We can also check that for $\tau = 1$ and for $\tau = 2$ we obtain a point whose image
  in~$G$ is not in the image of the Selmer group, so (*) is verified in this
  case, too.
\end{proof}

To conclude the proof of Theorem~\ref{T:11}, it now only remains to
observe that the $j$-invariants $21952/9$ of~$96a1$ and $1536$ of~$864c1$
are not good (the condition on the $3$-adic valuation is violated),
so the only remaining point in~$X_{96a1}(11)(\Q)$ and in~$X_{864c1}(11)(\Q)$
does not lead to a primitive integral solution of our Generalized Fermat
Equation. The only remaining point in~$X_{864b1}(11)(\Q)$ is the canonical
point; it corresponds to the Catalan solutions.

\begin{remark}
  According to work by Ligozat~\cite{Ligozat}, the Jacobian~$J(11)$ of~$X(11)$
  splits up to isogeny (and over~$\Q(\sqrt{-11})$) into a product of
  eleven copies of~$X_0(11)$, ten copies of a second elliptic curve
  and five copies of a third elliptic curve (which is $X_\ns^+(11) = 121b1$).
  The powers of these three
  elliptic curves correspond to isotypical components of the representation
  of the automorphism group of the $j$-map $X(11) \to \PP^1$ on the Lie
  algebra of~$J(11)$; the splitting into these three powers therefore
  persists over~$\Q$ after twisting. Our approach uses the $11$-dimensional
  factor (it can be identified with the kernel of the trace map from
  $R_{K_E/\Q} X_0(11)_{K_E}$ to~$X_0(11)$, up to isogeny). It would be nice if one could
  use the $5$-dimensional factor instead in the hope of eliminating the
  dependence on~GRH, but so far we have not found a description that
  would allow us to work over a smaller field.
\end{remark}


\section{The Generalized Fermat Equation with exponents 2, 3, 13} \label{S:p=13}

In this section, we collect some partial results for the case $p = 13$.
More precisely, we show that the Frey curve associated to any putative
solution must have irreducible $13$-torsion Galois module and that
only trivial solutions can be associated to the two CM curves
in the list of Lemma~\ref{L:7curves}.

\subsection{Eliminating reducible 13-torsion} \label{S:13red} \strut

The case $p = 13$ is special in the sense that it is a priori
possible to have Frey curves with reducible $13$-torsion Galois modules.
In this respect, it is similar to $p = 7$; compare~\cite{PSS2007}.
To deal with this possibility, we note that such a Frey curve~$E$
will have a Galois-stable subgroup~$C$ of order~$13$ and so gives rise
to a rational point~$P_E$ on~$X_0(13)$, which is a curve of genus~$0$.
The Galois action on~$C$ is via some character $\chi \colon G_{\Q} \to \F_{13}^\times$,
which can be ramified at most at~$2$, $3$ and~$13$. Associated to~$\chi$
is a twist~$X_\chi(13)$ of~$X_1(13)$ that classifies elliptic curves
with a cyclic subgroup of order~$13$ on which the Galois group acts via~$\chi$;
the Frey curve~$E$ corresponds to a rational point on~$X_\chi(13)$
that maps to~$P_E$ under the canonical covering map $X_\chi(13) \to X_0(13)$.
The covering $X_1(13) \to X_0(13)$ is Galois of degree~$6$ with Galois group naturally
isomorphic to $\F_{13}^\times/\{\pm 1\}$; the coverings $X_\chi(13) \to X_0(13)$
are twisted forms of it, corresponding to the composition
\[ G_{\Q} \stackrel{\chi}{\To} \F_{13}^\times \To \F_{13}^\times/\{\pm 1\} \simeq \Z/6\Z \,, \]
which is an element of $H^1(\Q, \Z/6\Z; \{2,3,13\})$
(where $H^1(K, M; S)$ denotes the subgroup of~$H^1(K, M)$ of cocycle classes
unramified outside~$S$). We can describe this group in the form
\begin{align*}
  H^1(\Q, \Z/6\Z; \{2,3,13\})
    &= H^1(\Q, \Z/2\Z; \{2,3,13\}) \oplus H^1(\Q, \Z/3\Z; \{2,3,13\}) \\
    &\simeq \langle -1, 2, 3, 13 \rangle_{\Q^\times/\Q^{\times 2}}
           \oplus \langle \omega, \tfrac{4+\omega}{3-\omega}
                        \rangle_{\Q(\omega)^\times/\Q(\omega)^{\times 3}} \,,
\end{align*}
where $\omega$ is a primitive cube root of unity. One can check that
a model of~$X_1(13)$ is given by
\[ y^2 = (v+2)^2 + 4\,, \qquad z^3 - v z^2 - (v+3) z - 1 = 0\,; \]
the map to $X_0(13) \simeq \PP^1$ is given by the $v$-coordinate.
The second equation can be written in the form
\[ \Bigl(\frac{z - \omega}{z - \omega^2}\Bigr)^3 = \frac{v - 3\omega}{v - 3\omega^2} \,, \]
which shows that it indeed gives a cyclic covering of~$\PP^1_v$
by~$\PP^1_z$. If $d$ is a squarefree integer representing an element
in~$\langle -1, 2, 3, 13 \rangle_{\Q^\times/\Q^{\times 2}}$ and $\gamma$
represents an element of~$\langle \omega, \tfrac{4+\omega}{3-\omega}
                        \rangle_{\Q(\omega)^\times/\Q(\omega)^{\times 3}}$,
then the corresponding twist is
\[ X_\chi(13) \colon d y^2 = (v+2)^2 + 4\,, \qquad
    \gamma \Bigl(\frac{z - \omega}{z - \omega^2}\Bigr)^3 = \frac{v - 3\omega}{v - 3\omega^2} \,.
\]
We note that the first equation defines a conic that has no real points when $d < 0$
and has no $3$-adic points when $3 \mid d$. This restricts us to $d \in \{1,2,13,26\}$.
We find hyperelliptic equations for the $36$~remaining curves (recall that
$X_1(13)$ has genus~$2$). It turns out that only eight of them have $\ell$-adic
points for $\ell \in \{2,3,13\}$. We list them in Table~\ref{Table:curves13a}.
In the table we give $d$ and~$\delta$, where $\gamma = \delta/\bar{\delta}$ and
the bar denotes the non-trivial automorphism of~$\Q(\omega)$. We denote the curve
in row~$i$ of the table by~$C_i$.

\begin{table}[htb]
\renewcommand{\arraystretch}{1.25}
\[ \begin{array}{|c|c|c|c|} \hline
    \text{no.} &  d &        \delta & f \\ \hline
             1 &  1 &             1 & x^6 - 2 x^5 + x^4 - 2 x^3 + 6 x^2 - 4 x + 1 \\
             2 &  2 &        \omega & 16 x^6 + 24 x^5 + 18 x^4 + 76 x^3 + 138 x^2 + 72 x + 16 \\
             3 &  2 &    \omega + 4 & 208 x^6 - 312 x^5 + 234 x^4 - 988 x^3 + 1794 x^2 - 936 x + 208 \\
             4 &  2 & -3 \omega - 4 & 16 x^6 - 24 x^5 + 106 x^4 - 252 x^3 + 226 x^2 - 72 x + 16 \\\hline
             5 & 13 &  3 \omega - 1 & x^6 + 2 x^5 + x^4 + 2 x^3 + 6 x^2 + 4 x + 1 \\
             6 & 26 &    \omega + 4 & 16 x^6 - 24 x^5 + 18 x^4 - 76 x^3 + 138 x^2 - 72 x + 16 \\
             7 & 26 &        \omega & 208 x^6 + 312 x^5 + 234 x^4 + 988 x^3 + 1794 x^2 + 936 x + 208 \\
             8 & 26 & -3 \omega - 4 & 16 x^6 - 24 x^5 + 106 x^4 - 252 x^3 + 226 x^2 - 72 x + 16 \\
             \hline
   \end{array}
\]

\medskip

\caption{Curves $X_\chi(13)$ with local points, given as $y^2 = f(x)$.}
\label{Table:curves13a}
\end{table}

We see that the last four curves are isomorphic to the first four. This is
because of the canonical isomorphism $X_1(13) \simeq X_\mu(13)$, where the
latter classifies elliptic curves with a subgroup isomorphic to~$\mu_{13}$.
On the level of~$X_0(13)$, this comes from the Atkin-Lehner involution,
which in terms of our coordinate~$v$ is given by $v \mapsto (v+12)/(v-1)$.

The first curve~$C_1$ is~$X_1(13)$; it is known that its Jacobian has Mordell-Weil
rank zero and that the only rational points on~$X_1(13)$ are six cusps
(there are no elliptic curves over~$\Q$ with a rational point of order~$13$).
For the curves $C_2$, $C_3$ and~$C_4$, a $2$-descent on the Jacobian as in~\cite{Stoll2001}
gives an upper bound of~$2$ for the rank. $C_2$ and~$C_4$
each have six more or less obvious rational points; their differences
generate a subgroup of rank~$2$ of the Mordell-Weil group, so their Jacobians
indeed have rank~$2$. On~$C_3$ one does not find small rational
points, and indeed it turns out that its $2$-Selmer set is empty, which
proves that it has no rational points. See~\cite{BS2009} for how to compute
the $2$-Selmer set. It remains to consider $C_2$ and~$C_4$.

We note that the $j$-invariant map on~$\PP^1_v \simeq X_0(13)$ is given by
\[ j = \frac{(v^2 + 3 v + 9)(v^4 + 3 v^3 + 5 v^2 - 4 v - 4)^3}{v-1} \,. \]
The obvious orbits of points on the six curves that do have rational points then give
points on~$X_0(13)$ with $v = \infty$, $0$, $-4$, $1$, $-12$, $-8/5$ and
$j$-invariants
\[ \infty\,, \quad \frac{12^3}{3}\,, \quad -\frac{12^3 \cdot 13^4}{5}\,, \quad \infty\,, \quad
  -\frac{12^3 \cdot 4079^3}{3}\,, \quad -\frac{12^3 \cdot (17 \cdot 29)^3 \cdot 13}{5^{13}}\,,
\]
respectively. None of these correspond to primitive solutions of $x^2 + y^3 = z^{13}$,
except $j = \infty$, which is related to the trivial solutions $(\pm 1, -1, 0)$.
So to rule out solutions whose Frey curves have reducible $13$-torsion,
it will suffice to show that there are no rational points on $C_2$ and~$C_4$
other than the orbit of six points containing the points at infinity.

Computing the $2$-Selmer sets, we find in both cases that its elements
are accounted for by the points in the known orbit. So in each orbit of rational points
under the action of the automorphism group, there is a point that lifts to the $2$-covering
of the curve that lifts the two points at infinity. So it is enough to look
at rational points on this $2$-covering.

We first consider~$C_2$. Its polynomial~$f$ splits off three
linear factors over~$K$, where $K$ is the field obtained by adjoining
one of the roots~$\alpha$ of~$f$ to~$\Q$. The relevant $2$-covering then maps
over~$K$ to the curve $y^2 = (x - \alpha) g(x)$, where $g$ is the remaining
cubic factor. This is an elliptic curve (with two $K$-points at infinity
and one with $x = \alpha$). Computing its $2$-Selmer group (this involves
obtaining the class group of a number field of degree~$18$, which we can
do without assuming GRH; the computation took a few days), we find that
it has rank~$1$. We know three $K$-points on the elliptic curve; they map
surjectively onto the Selmer group. So we can do an Elliptic Curve Chabauty
computation, which tells us that the only $K$-points whose $x$-coordinate
is rational are the two points at infinity. This in turn implies that
the known rational points on~$C_2$ are the six points in the orbit of the
points at infinity.

Now we consider~$C_4$. Here the field generated by a
root of~$f$ is actually Galois (with group~$S_3$). We work over its cubic subfield~$L$.
Over~$L$, $f$ splits as $16$~times the product of three monic quadratic
factors $h_1$, $h_2$, $h_3$,
and we consider the elliptic curve~$E$ given as $y^2 = h_1(x) h_2(x)$,
with one of the points at infinity as the origin.
This curve has full $2$-torsion over~$L$, so a $2$-descent is easily
done unconditionally. We find that the $2$-Selmer group has rank~$3$
and that the difference of the two points at infinity has infinite order,
so the Mordell-Weil rank of~$E$ over~$L$ is~$1$. An Elliptic Curve Chabauty
computation then shows that the only $K$-points on~$E$ with rational
$x$-coordinate are those at infinity and those with $x$-coordinate~$-3$.
Since there are no rational points on~$C_4$ with $x$-coordinate~$-3$,
this shows as above for~$C_2$ that the only rational points are
the six points in the orbit of the points at infinity.

This proves the following statement.

\begin{lemma} \label{L:13nored}
  Let $(a,b,c) \in \Z^3$ be a non-trivial primitive solution of $x^2 + y^3 = z^{13}$.
  Then the $13$-torsion Galois module $E_{(a,b,c)}[13]$ of the associated
  Frey curve is irreducible.
\end{lemma}

\subsection{Dealing with the CM curves} \strut

$13$ is congruent to~$1$ both mod~$3$ and mod~$4$, so the $13$-torsion
Galois representations on~$27a1$ and on~$288a1$ both have image contained
in the normalizer of a split Cartan subgroup. But unfortunately the general
result of~\cite{BPR2013} does not apply in this case. We can, however,
use the approach taken in Section~\ref{S:13red} above. Since we are in
the split case, the curves have cyclic subgroups of order~$13$ defined
over a quadratic field~$K$, which is $\Q(\omega)$ for~$27a1$ (with $\omega$
a primitive cube root of unity) and~$\Q(i)$ for~$288a1$. We find the
twist of~$X_1(13)$ over~$K$ that corresponds to the Galois representation
over~$K$ on this cyclic subgroup. Finding the twist is not entirely trivial,
since the points on~$X_0(13)$ corresponding to $27a1$ or to~$288a1$ are
branch points for the covering $X_1(13) \to X_0(13)$ (of ramification degree~$3$,
respectively~$2$). In the case of~$27a1$ we use a little trick: the isogenous
curve~$27a2$ has isomorphic Galois representation, but $j$-invariant $\neq 0$,
so the corresponding point in~$X_0(13)(K)$ lifts to a unique twist, which
must be the correct one also for~$27a1$. Since cube roots of unity are
in~$K$, we can make a coordinate change so that the automorphism of order~$3$
is given by multiplying the $x$-coordinate by~$\omega$. We obtain the following
simple model over $K = \Q(\omega)$ of the relevant twist of~$X_1(13)$:
\[ C_{27a1} \colon y^2 = x^6 + 22 x^3 + 13 \,. \]
The points coming from~$27a1$ are the two points at infinity, and the points
coming from~$27a2$ are the six points whose $x$-coordinate is a cube root
of~unity.

For $288a1$, we figure out the quadratic part of the sextic twist (the cubic
part is unique in this case) by looking at the Galois action on the cyclic
subgroup explicitly. We find that the correct twist of~$X_1(13)$ is
\[ C_{288a1} \colon y^2 = 12 i x^5 + (30 i + 33) x^4 + 66 x^3 + (-30 i + 33) x^2 - 12 i x \,. \]
Here the points coming from~$288a1$ are the ramification points $(0,0)$
and~$(-1,0)$ and the (unique) point at infinity. There are (at least) six further
points over~$\Q(i)$ on this curve, forming an orbit under the automorphism
group, of which $\bigl((4 i - 3)/6, 35/36\bigr)$ is a representative.

As a first step, we compute the $2$-Selmer group of the Jacobian~$J$ of each
of the two curves. In both cases, we find an upper bound of~$2$ for the
rank of~$J(K)$. The differences of the known points on the curve generate
a group of rank~$2$, so we know a subgroup of finite index of~$J(K)$.
It is easy to determine the torsion subgroup, which is $\Z/3\Z$ for~$C_{27a1}$
and $\Z/2\Z \times \Z/2\Z$ for~$C_{288a1}$. Using the reduction modulo several
good primes of~$K$, we check that our subgroup is saturated at the primes
dividing the group order of the reductions for the primes above~$7$ for~$C_{27a1}$,
or the primes above~$5$ for~$C_{288a1}$. We also use this reduction information
for a bit of Mordell-Weil sieving (compare~\cite{BS2010}) to show that
any point in~$C_{27a1}(\Q(\omega))$ with rational $j$-invariant must
reduce modulo both primes above~$7$ to the image of one of the known eight points
(the same for both primes),
and that any point in~$C_{288a1}(\Q(i))$ with rational $j$-invariant
must reduce modulo both primes above~$5$ to the image of one of the three
points coming from~$288a1$ (the other six points have $j$ in~$\Q(i) \setminus \Q$).

It remains to show that these points are the only points in their residue
classes mod~$7$, respectively, mod~$5$. For this, we use the criterion
in~\cite{Siksek2013}*{Theorem~2}. We compute the integrals to sufficient
precision and then check that the pair of differentials killing~$J(K)$
is `transverse' mod~$7$ (or~$5$) at each of the relevant points, which
comes down to verifying the assumption in Siksek's criterion. Note that
we apply Chabauty's method for a genus~$2$ curve when the rank is~$2$;
this is possible because we are working over a quadratic field. See the
discussion in~\cite{Siksek2013}*{Section~2}.

We obtain the following result.

\begin{lemma} \label{L:13noCM}
  Let $(a,b,c) \in \Z^3$ be a non-trivial primitive solution of $x^2 + y^3 = z^{13}$.
  Then the $13$-torsion Galois module $E_{(a,b,c)}[13]$ of the associated
  Frey curve is, up to quadratic twist, symplectically isomorphic to~$E[13]$ for some
  $E \in \{96a2,864a1,864b1,864c1\}$.
\end{lemma}

\begin{proof}
  By Lemma~\ref{L:13nored}, $E_{(a,b,c)}[13]$ is irreducible, so by
  Theorem~\ref{T:nominus}, it is symplectically isomorphic to~$E[13]$, where
  $E$ is one of the given curves or one of the CM~curves $27a1$, $288a1$
  or~$288a2$. These latter three are excluded by the computations reported
  on above.
\end{proof}

A \Magma\ script that performs the necessary computations for the results
in this section is available as~\texttt{section8.magma} at~\cite{programs}.
It relies on \texttt{X1\_13\_opt.magma}, which is available at the same location.


\section{Possible extensions}

Since $X_0(17)$ and~$X_0(19)$ are elliptic curves like~$X_0(11)$, the approach
taken in this paper to treat the case $p = 11$ has a chance to also work
for $p = 17$ and/or $p = 19$. Since $X_0(17)$ has a rational point of order~$4$
and $X_0(19)$ has a rational point of order~$3$, suitable Selmer groups
can be computed with roughly comparable effort (requiring arithmetic information
of number fields of degree~$36$ or~$40$). This is investigated in ongoing work by the
authors. For $p = 13$, results beyond those obtained in the preceding section
are likely harder to obtain than for $p = 17$ or $p = 19$, since $X_0(13)$
is a curve of genus~$0$. We can try to work with twists of~$X_1(13)$,
which is a curve of genus~$2$; our approach would require us to find
the (relevant) points on such a twist over a field of degree~$14$.
The standard method of $2$-descent on the Jacobian of such a curve would
require working with a number field of degree $6 \cdot 14 = 84$, which is
beyond the range of feasibility of current algorithms (even assuming GRH).


\end{document}